\documentclass[11pt,a4paper,dvipdfmx]{article}
\usepackage{amsmath,amsthm,amssymb,color,enumerate,graphicx,hyperref,mathrsfs}
\usepackage[numbers]{natbib}

\def\A{\mathcal{A}}

\def\L{\mathcal{L}}
\def\M{\mathcal{M}}
\def\N{\mathbb{N}}

\def\R{\mathbb{R}}

\theoremstyle{plain}
\newtheorem{thm}{Theorem}[section]

\newtheorem{lem}{Lemma}[section]
\newtheorem{prop}{Proposition}[section]

\theoremstyle{definition}
\newtheorem{assmp}{Assumption}[section]

\newtheorem{exmp}{Example}[section]
\newtheorem*{exmp'}{Example 3.2$'$}

\newtheorem{rem}{Remark}[section]

\numberwithin{equation}{section}
\numberwithin{figure}{section}
\allowdisplaybreaks

\title{Value Functions and Optimality Conditions for Nonconvex Variational Problems with an Infinite Horizon in Banach Spaces\thanks{The authors are grateful to the two anonymous referees for their helpful comments and suggestions on the earlier  version of this manuscript.}}
\date{\today}
\author{H\'el\`ene Frankowska\thanks{The research of this author was supported by the Air Force Office of Scientific Research, USA under award number FA9550-18-1-0254. It also partially benefited from the FJMH Program PGMO and from the support to this program from EDF-THALES-ORANGE-CRITEO under grant PGMO 2018-0047H.} \\{\normalsize CNRS Institut de Math\'ematiques de Jussieu -- Paris Rive Gauche} \\ {\normalsize Sorbonne Universit\'e, Campus Pierre et Marie Curie} \\ {\normalsize Case 247, 4 Place Jussieu, 75252 Paris, France} \\ {\small e-mail: helene.frankowska@imj-prg.fr}
\and \\
Nobusumi Sagara\thanks{The research of this author benefited from the support of JSPS KAKENHI Grant Number JP18K01518 from the Ministry of Education, Culture, Sports, Science and Technology, Japan.} \\
{\normalsize Department of Economics, Hosei University} \\
{\normalsize 4342, Aihara, Machida, Tokyo, 194--0298, Japan} \\
{\small e-mail: nsagara@hosei.ac.jp}}

\begin{document}
\maketitle
\setcounter{page}{0}
\thispagestyle{empty}
\clearpage

\begin{abstract} 
We investigate the value function of an infinite horizon variational problem in the infinite-dimensional setting. Firstly, we provide an upper estimate of its Dini--Hadamard subdifferential in terms of the Clarke subdifferential of the Lipschitz continuous integrand and the Clarke normal cone to the graph of the set-valued mapping describing dynamics. Secondly, we derive a necessary condition for optimality in the form of an adjoint inclusion that grasps a connection between the Euler--Lagrange condition and the maximum principle. \\


\noindent
{\bfseries MSC2010 Subject Classification:} Primary: 34A60, 49J50, 49J52; Secondary: 49J53, 49K15, 90C39 \\

\noindent
{\bfseries OR/MS Subject Classification:} Dynamic programming/optimal control: Applications; Deterministic; Programming: Nondifferentiable
\end{abstract}


\section{Introduction}
Optimal control and dynamic programming are instrumental cornerstones of modern economic growth theory originated in \citet{ra28}. In the general reduced model of capital accumulation, necessary (and sufficient) conditions for optimality are employed under the convexity assumptions on utility functions and technologies for the investigation of the existence of competitive equilibria and support prices; see \citet{bs82,ma82,ta82,ta84}. Such well-behaving properties are prominent in convex problems of optimal control explored in the classical work by \citet{ro70} with the full power of duality theory in convex analysis. In particular, one of the  advantages in convex economic models lies in the crucial observation that the differentiability of the value function is guaranteed under the smoothness assumptions on the data; see \citet{bs79,bs82,blv96,rzs12,ta84}.  

On the contrary, the absence of convexity and smoothness are two major sources of complex economic dynamics in continuous time as illustrated in \citet{alv99,dh81,hk03,sk78,wa03}. More to the point, the difficulty with the lack of convexity assumptions results in the failure of differentiability of the value function even if the underlying data are smooth. Without convexity, one can expect at best the Lipschitz continuity of the value function even for smooth problems. This causes problems with expressing optimality conditions in many nonconvex economic growth models when one attempts to apply the Hamilton--Jacobi--Bellman (HJB) equation. Recall that the value function is its unique solution whenever it is smooth. 

The well-known failure of differentiability of the value function has stimulated two alternative approaches in optimal control theory. One is the application of a ``generalized'' subdifferential calculus along the lines of \citet{cl83}, which eventually leads to the formulation of a relation between the maximum principle and dynamic programming whenever the value function is locally Lipschitz continuous; see \citet{cv83,cv87}. The other independent development is the concept of ``viscosity solutions'' to the HJB equation initiated by \citet{li82} (see also \citet{cel84,crl83}), which makes use of the notion of Fr\'echet super- and subdifferentials to claim that the value function is the unique viscosity solution of the HJB equation. For the connections between the maximum principle and the superdifferentials of the value function, see \citet{fr89a,fr02}. 

With this background in mind, we investigate the value function of an infinite horizon variational problem in the setting of an infinite-dimensional generalized control system. Our primary concern here is to go beyond convexity, smoothness, and finite dimensionality aiming the possible applications to dynamic optimization in economic theory. Since the optimal economic growth models are identified with a specific form of the general equilibrium model with single representable consumer and firm, we can deal with a rich class of commodity spaces for capital stock, which appears as a Sobolev space. In particular, spatial Ramsey growth models involve a location of each agent along the lines of \citet{ho29}, in which infinite-dimensional commodity spaces naturally arise; see \citet{bcf13,bcz09,br04,bxy14}. Applying our general result, we obtain another necessary condition for optimality in spatial Ramsey growth models.

The purpose of this paper is twofold. Firstly, we provide an upper estimate of the Dini--Hadamard subdifferential of the value function in terms of the Clarke subdifferential of the Lipschitz continuous integrand and the Clarke normal cone to the set-valued mapping describing dynamics. As a result, we obtain the strict differentiability of the value function under the Fr\'echet differentiability of the integrand, which removes completely the convexity assumptions of the earlier works by \citet{bs79,bs82,blv96,rzs12,ta84}. For the (sub)differentiability of the value function in the context of finite dimensional control systems with a finite horizon, see the lecture notes \citet{fr02}. 

Secondly, under an interiority assumption we derive a necessary condition for optimality in the form of an adjoint inclusion that grasps a connection between the Euler--Lagrange condition and the maximum principle. Our interiority assumption is weaker than those in \citet{bs79,bs82,blv96,ta82,ta84}. On the other hand, when dynamics are described by a control system, such interiority assumption may be omitted. To deal with the adjoint variable in dual spaces, we introduce the Gelfand integrals of the Dini--Hadamard and Clarke subdifferential mappings, which is a new feature that does not arise in the context of finite-dimensional control systems. 

For the finite-dimensional control systems, necessary conditions with or without convexity assumptions using limiting subdifferentials were obtained in \citet{io97,vz97} in the finite horizon setting. The ones in the infinite horizon setting using Dini--Hadamard, Clarke, and limiting subdifferentials were derived in \citet{ac79,cf18,sa10,ye93}. For control systems in Hilbert spaces, a necessary condition under the convexity assumptions was obtained in \citet{ba78b} in the infinite horizon setting. For Banach spaces and semilinear control systems, the necessary and sufficient conditions were derived in \citet{cf92} in the finite horizon setting when the set of velocities is convex. 

The organization of the paper is as follows. Section 2 collects preliminary results on subdifferential calculus on Banach spaces. In Section 3 we formulate the nonconvex variational problem under investigation with the standing hypotheses and demonstrate the Lipschitz continuity and subdifferentiability of the value function. We derive in Section 4 necessary conditions for the variational and optimal control problems. Section 5 applies our main result to spatial Ramsey growth models. Appendices I and II discuss Gelfand integral of multifunctions and the Gelfand integrability of the Dini--Hadamard and Clarke subdifferential mappings, and the proofs of auxiliary results and lemmas needed to obtain the main results.

\section{Preliminaries}
Let $(E,\| \cdot \|)$ be a real Banach space with the dual system $\langle E^*,E \rangle$, where $E^*$ is the norm dual of $E$. A real-valued function $\varphi:E\to \R$ is said to be \textit{Gateaux differentiable} at $\bar{x}\in E$ if there exists an element $\nabla\varphi(\bar{x})\in E^*$ such that
\begin{equation}
\label{eq1}
\lim_{\theta\to 0}\frac{\varphi(\bar{x}+\theta v)-\varphi(\bar{x})}{\theta}=\langle \nabla\varphi(\bar{x}),v \rangle
\end{equation}
for every $v\in E$; $\nabla\varphi(\bar{x})$ is called the \textit{Gateaux derivative} of $\varphi$ at $\bar{x}$. If the convergence in \eqref{eq1} is uniform in $v\in C$ for every bounded subset $C$ of $E$, then $\varphi$ is said to be \textit{Fr\'echet differentiable} at $\bar{x}$ and $\nabla\varphi(\bar{x})$ is called the \textit{Fr\'echet derivative} of $\varphi$ at $\bar{x}$. A function $\varphi$ is said to be \textit{strictly differentiable} at $\bar{x}$ if there exists $\nabla\varphi(\bar{x})\in E^*$ such that
\begin{equation}
\label{eq2}
\lim_{\begin{subarray}{c} x\to \bar{x} \\ \theta\to 0 \end{subarray}}\frac{\varphi(x+\theta v)-\varphi(x)}{\theta}=\langle \nabla\varphi(\bar{x}),v \rangle
\end{equation}
and the convergence in \eqref{eq2} is uniform in $v\in C$ for every compact subset $C$ of $E$. Then $\nabla\varphi(\bar{x})$ is called the \textit{strict derivative} of $\varphi$ at $\bar{x}$. If $\varphi$ is strictly differentiable at $\bar{x}$, then $\varphi$ is Lipschitz near $\bar{x}$; see \citet[Proposition 2.2.1]{cl83}. A function $\varphi$ is said to be  \textit{continuously differentiable} at $\bar{x}$ if $\varphi$ is Gateaux differentiable at every $x$ in a neighborhood $O$ of $\bar{x}$ and the mapping $x\mapsto \nabla\varphi(x)$ is continuous from $O$ to $E^*$; $\varphi$ is called a \textit{$C^1$-function} on $E$ if $\varphi$ is continuously differentiable at any point in $E$. If $\varphi$ is continuously differentiable at $\bar{x}$, then $\varphi$ is strictly differentiable at $\bar{x}$; see \citet[Corollary, p.\,32]{cl83}. A norm $\| \cdot \|$ on a Banach space $E$ is said to be Gateaux (resp.\ Fr\'echet) differentiable if $\| \cdot \|$ is Gateaux (resp.\ Fr\'echet) differentiable on the open set $E\setminus \{ 0 \}$. 

The \textit{support function} $s(\cdot,C):E^*\to \R\cup \{+\infty \}$ of a nonempty subset $C$ of $E$ is given by $s(x^*,C)=\sup_{x\in C}\langle x^*,x \rangle$. The \textit{polar} $C^0$ of $C$ is the set $C^0=\{ x^*\in E^*\mid s(x^*,C)\le 0 \}$. The support function $s:(\cdot, K):E\to \R\cup \{+\infty \}$ of a nonempty subset $K$ of $E^*$ is defined by $s(x, K)=\sup_{x^*\in K}\langle x^*,x \rangle$. The polar $K^0$ of $K$ is the set $K^0=\{ x\in E\mid s(x,K)\le 0 \}$. 

Let $\varphi:E\to \R\cup \{ +\infty \}$ be an extended real-valued function on $E$. The \textit{effective domain} of $\varphi$ is the set of points where $\varphi$ is finite and is denoted by $\mathrm{dom}\,\varphi:=\{ x\in E\mid \varphi(x)<+\infty\}$. If $\varphi$ is Lipschitz near $\bar{x}\in \mathrm{dom}\,\varphi$, then its \textit{Clarke directional derivative} at $\bar{x}$ in the direction $v\in E$ is defined by
$$
\varphi^\circ(\bar{x};v):=\limsup_{\begin{subarray}{c} x\to\bar{x} \\ \theta\downarrow 0 \end{subarray}}\frac{\varphi(x+\theta v)-\varphi(x)}{\theta}
$$
and the \textit{Clarke subdifferential} of $\varphi$ at $\bar{x}$ is defined by 
$$
\partial^\circ \varphi(\bar{x}):=\{ x^*\in E^*\mid \langle x^*,v\rangle\le \varphi^\circ(\bar{x};v)\ \forall v\in E \}. 
$$
Since the function $v\mapsto \varphi^\circ(\bar{x};v)$ is positively homogeneous and subadditive, the set $\partial^\circ \varphi(\bar{x})$ is nonempty by the Hahn--Banach theorem, weakly$^*\!$ compact, and convex in $E^*$. Furthermore, the Clarke directional derivative is the support function of the Clarke subdifferential
$$
\varphi^\circ(\bar{x};v)=s(v,\partial^\circ \varphi(\bar{x}))
$$ 
for every $v\in E$; see \citet[Propositions 2.1.1 and 2.1.2]{cl83}. Recall that the function $\varphi$ that is Lipschitz near $\bar{x}$ is said to be \textit{regular} at $\bar{x}$ if the classical directional derivative
$$
\varphi'(\bar{x};v):=\lim_{\theta\downarrow 0}\frac{\varphi(\bar{x}+\theta v)-\varphi(\bar{x})}{\theta}
$$
exists and $\varphi'(\bar{x};v)=\varphi^\circ(\bar{x};v)$ for every $v\in E$. 

Let $d_C:E\to \R$ be the \textit{distance function} from a nonempty subset $C$ of $E$ defined by $d_C(x):=\inf_{\xi\in C}\| x-\xi \|$. Then $d_C$ is nonexpansive (i.e., Lipschitz of rank one) on $E$. Let $\bar{x}$ be a point in $C$. A vector $v\in E$ is called a \textit{tangent} to $C$ at $\bar{x}$ if $d_C^\circ(\bar{x};v)=0$. The set of all tangents to $C$ at $\bar{x}$ is called the \textit{Clarke tangent cone} to $C$ at $\bar{x}$ and is denoted by 
$$
T_C(\bar{x}):=\{ v\in E\mid d_C^\circ(\bar{x};v)=0 \}.
$$
Then $T_C(\bar{x})$ is a closed convex cone because $v\mapsto d_C^\circ(\bar{x};v)$ is nonnegative, positively homogeneous, and continuous. An intrinsic characterization of $T_C(\bar{x})$ that is independent of the use of a distance function is as follows: $v\in T_C(\bar{x})$ if and only if for every sequence $\{ x_n \}_{n\in \N}$ in $C$ with $x_n\to \bar{x}$ and every sequence $\{ \theta_n \}_{n\in \N}$ of positive real numbers with $\theta_n\downarrow 0$, there is a sequence $\{ v_n \}_{n\in \N}$ in $E$ with $v_n\to v$ such that $x_n+\theta_nv_n\in C$ for each $n\in \N$; see \citet[Theorem 2.4.5]{cl83}. Let $B$ be the open unit ball in $E$. Define the \textit{contingent cone} $K_C(\bar{x})$ of tangents to $C$ at $\bar{x}$ by
$$
K_C(\bar{x}):=\left\{ v\in E \mid \forall \varepsilon>0\,\exists\theta\in (0,\varepsilon)\,\exists w\in v+\varepsilon B: \bar{x}+\theta w\in C \right\}.
$$   
Then $v\in K_C(\bar{x})$ if and only if there exist a sequence $\{ \theta_n \}_{n\in \N}$ of positive real numbers with $\theta_n\downarrow 0$ and a sequence $\{ v_n \}_{n\in \N}$ in $E$ with $v_n\to v$ such that $\bar{x}+\theta_nv_n\in C$ for each $n\in \N$. It is evident that $T_C(\bar{x})\subset K_C(\bar{x})$, but $K_C(\bar{x})$ is not necessarily convex. The set $C$ is said to be \textit{regular} at $\bar{x}$ if $T_C(\bar{x})=K_C(\bar{x})$. The polar of $T_C(\bar{x})$ is called the \textit{Clarke normal cone} to $C$ at $\bar{x}$, which is given by 
$$
N_C(\bar{x})=\{ x^*\in E^* \mid \langle x^*,v \rangle\le 0 \ \forall v\in T_C(\bar{x}) \}. 
$$
The Clarke normal cone is characterized by $N_C(\bar{x})=\mathit{w}^*\text{-}\mathrm{cl}\{ \bigcup_{\lambda\ge 0}\lambda\partial^\circ d_C(\bar{x}) \}$ (see \citet[Proposition 2.4.2]{cl83}), where the right-hand side of the above equality means the weak$^*\!$ closure of the set. It follows from the bipolar theorem (see \citet[Theorem 2.4.3]{af90}) that $T_C(\bar{x})$ is the polar of $N_C(\bar{x})$, i.e., $T_C(\bar{x})=\{ v\in E\mid \langle x^*,v \rangle \le 0\ \forall x^*\in N_C(\bar{x}) \}$. Denote by $\mathrm{epi}\,\varphi=\{ (x,r)\in E\times \R\mid \varphi(x)\le r \}$ the \textit{epigraph} of $\varphi$. If $\varphi$ is Lipschitz near $\bar{x}\in \mathrm{dom}\,\varphi$, then $T_{\mathrm{epi}\,\varphi}(\bar{x},\varphi(\bar{x}))=\mathrm{epi}\,\varphi^\circ(\bar{x};\cdot)$ (see \citet[Theorem 2.4.9]{af90}), and hence, $\varphi^\circ(\bar{x};v)=\inf\{ r\in \R\mid (v,r)\in T_{\mathrm{epi}\,\varphi}(\bar{x},\varphi(\bar{x})) \}$.   
Therefore, if $\varphi$ is Lipschitz near $\bar{x}$, then 
$$
\partial^\circ\varphi(\bar{x})=\{ x^*\in E^*\mid (x^*,-1)\in N_{\mathrm{epi}\,\varphi}(\bar{x},\varphi(\bar{x})) \}.
$$

The \textit{lower directional derivative} (or \textit{contingent}, or \textit{Dini--Hadamard directional subderivative}) of $\varphi:E\to \R\cup \{ +\infty \}$ at $\bar{x}\in \mathrm{dom}\,\varphi$ in the direction $v\in E$ is defined by  
$$
\varphi^-(\bar{x};v):=\liminf_{\begin{subarray}{c}u\to v \\ \theta\downarrow 0 \end{subarray}}\frac{\varphi(\bar{x}+\theta u)-\varphi(\bar{x})}{\theta}\in \R\cup \{ \pm\infty \}
$$
and the \textit{upper directional derivative} (or \textit{Dini--Hadamard directional super\-derivative}) of $\varphi$ at $\bar{x}$ in the direction $v\in E$ is defined by  
$$
\varphi^+(\bar{x};v):=\limsup_{\begin{subarray}{c}u\to v \\ \theta\downarrow 0 \end{subarray}}\frac{\varphi(\bar{x}+\theta u)-\varphi(\bar{x})}{\theta}\in \R\cup \{ \pm\infty \}.
$$
The \textit{Dini--Hadamard subdifferential} of $\varphi$ at $\bar{x}$ is defined by
$$
\partial^-\varphi(\bar{x}):=\{ x^*\in E^*\mid \langle x^*,v \rangle \le \varphi^-(\bar{x};v) \ \forall v\in E \}
$$
and the \textit{Dini--Hadamard superdifferential} of $\varphi$ at $\bar{x}$ is defined by
$$
\partial^+\varphi(\bar{x}):=\{ x^*\in E^*\mid \langle x^*,v \rangle \ge \varphi^+(\bar{x};v) \ \forall v\in E \}.
$$
Because of the plus-minus symmetry with $\varphi^-(x;v)=-(-\varphi)^+(x;v)$ and $\partial^-\varphi(\bar{x})=-\partial^+(-\varphi)(\bar{x})$, it is enough to investigate lower directional derivatives and Dini--Hadamard subdifferentials in what follows. Since 
$$
K_{\mathrm{epi}\,\varphi}(\bar{x},\varphi(\bar{x}))=\mathrm{epi}\,\varphi^-(\bar{x};\cdot)
$$ 
(see \citet[Propositions 6.1.3 and 6.1.4]{af90}), 
$$
\varphi^-(\bar{x};v)=\inf\{ r\in \R\mid (v,r)\in K_{\mathrm{epi}\,\varphi}(\bar{x},\varphi(\bar{x})) \}\in \R\cup \{ \pm\infty \}
$$ 
with the convention that $\inf \emptyset=+\infty$. Therefore, if $\bar{x}\in \mathrm{dom}\,\varphi$, then 
$$
\partial^-\varphi(\bar{x})=\{ x^*\in E^*\mid (x^*,-1)\in K_{\mathrm{epi}\,\varphi}(\bar{x},\varphi(\bar{x}))^0 \}.
$$

Unlike Clarke directional derivatives, the lower directional derivative mapping $v\mapsto \varphi^-(\bar{x};v)$ fails to be convex although it is positively homogeneous. Thus, except for a smooth or a convex function $\varphi$, it is rather typical that $\partial^-\varphi(\bar{x})$ is empty at some points for a lower semicontinuous or even a locally Lipschitz function. Note that $\partial^-\varphi(\bar{x})$ is weakly$^*\!$ closed and convex. If $\varphi$ is locally Lipschitz, then $\varphi^-(\bar{x},v)\le \varphi^\circ(\bar{x};v)$ for every $v\in E$, and hence, $\partial^-\varphi(\bar{x})\subset \partial^\circ \varphi(\bar{x})$. In particular, if $\varphi$ is also regular at $\bar{x}$, then $\varphi^-(\bar{x},v)=\varphi^\circ(\bar{x};v)$ for every $v\in E$ and $\partial^-\varphi(\bar{x})=\partial^\circ \varphi(\bar{x})$. Note also that if $\varphi$ has the strict derivative $\nabla\varphi(\bar{x})$ at $\bar{x}\in E$, then $\partial^-\varphi(\bar{x})=\{ \nabla\varphi(\bar{x}) \}$. 

A generic existence of Dini--Hadamard subdifferentials is assured in the following result. 

\begin{thm}[\citet{io84,io17}]
\label{thm0}
Let $E$ be a Banach space admitting an equivalent Gateaux differentiable norm and $\varphi:E\to \R\cup \{ +\infty \}$ be a lower semicontinuous function. Then the set $\{ x\in E\mid \partial^-\varphi(x)\ne \emptyset \}$ is dense in $\mathrm{dom}\,\varphi$. 
\end{thm}

\noindent
We recall that any separable Banach space has an equivalent Gateaux differentiable norm; see \citet[Theorem 8.2]{fhhmz11}.

The next extends the well-known representation of the normal cone to the set determined by the inequality constraint. 

\begin{prop}
\label{exmp6}
Let $\varphi_i:E\to \R$, $i=1,2,\dots,m$, be continuous real-valued functions and 
$$
C:=\left\{ x\in E\mid \varphi_i(x)\le 0,\,i=1,2,\dots,m \right\}. 
$$ 
Define the active constraint indices at $\bar{x}\in C$ by $I(\bar{x}):=\{ i\in \{ 1,2,\dots,m \}\mid \varphi_i(\bar{x})=0 \}$. If $\varphi_i$ is strictly differentiable at $\bar{x}$ for each $i\in I(\bar{x})$ and the constraint qualification $0\not\in \mathrm{co}\,\{ \nabla \varphi_i(\bar{x})\mid i\in I(\bar{x}) \}$ is satisfied, then
$$
N_C(\bar{x})=\left\{ \sum_{i\in I(\bar{x})}\lambda_i\nabla \varphi_i(\bar{x})\in E^*\mid \lambda_i\ge 0 \ \forall i\in I(\bar{x}) \right\}.
$$
\end{prop}

\noindent
The proof is provided in Subsection \ref{apdx0} because we could not find this result in the literature for arbitrary Banach spaces.

\section{Value Functions for an Infinite Horizon Problem}
\subsection{Nonconvex Variational Problems}
Denote by $\R_+=[0,\infty)$ the unbounded interval of the real line with the Lebesgue measure and the Lebesgue $\sigma$-algebra $\L$. A function $x:\R_+\to E$ is said to be \textit{simple} if there exist $x_1,x_2,\dots,x_n\in E$ and $I_1,I_2,\dots,I_n\in \L$ such that $x(\cdot)=\sum_{i=1}^nx_i\chi_{I_i}$, where $\chi_{I_i}(t)=1$ if $t\in I_i$ and $\chi_{I_i}(t)=0$ otherwise. A function $x(\cdot)$ is said to be \textit{strongly measurable} if there exists a sequence of simple functions $\{ x_n(\cdot) \}_{n\in \N}$ from $\R_+$ to $E$ such that $\| x_n(t)-x(t) \|\to 0$ a.e.\ $t\in \R_+$. A strongly measurable function $x(\cdot)$ is \textit{locally Bochner integrable} if it is Bochner integrable on every compact subset $I$ of $\R_+$, that is, $\int_I\| x(t) \|dt<\infty$, where the \textit{Bochner integral} of $x(\cdot)$ over $I$ is defined by $\int_I x(t)dt:=\lim_n\int_I x_n(t)dt$. Let $L^1_{\mathrm{loc}}(\R_+,E)$ be the space of (the equivalence classes of) locally Bochner integrable functions from $\R_+$ to $E$. 

A function $x(\cdot):\R_+\to E$ is said to be \textit{strongly differentiable} at $t>0$ if there exists $v\in E$ such that 
$$
\lim_{h\to 0}\frac{x(t+h)-x(t)}{h}=v. 
$$
The vector $v$ is denoted by $\dot{x}(t)$ and called the \textit{strong derivative} of $x$ at $t$. Denote by $W^{1,1}_{\mathrm{loc}}(\R_+,E)$ the Sobolev space, which consists of locally Bochner integrable functions $x:\R_+\to E$ whose strong derivative $\dot{x}(t)$ exists a.e.\ $t\in \R_+\setminus \{ 0 \}$ with $\dot{x}(\cdot)\in L^1_{\mathrm{loc}}(\R_+,E)$ and $x(t)=\int_0^t\dot{x}(s)ds+x(0)$ for every $t\in \R_+$. For each $n\in \N$, define the seminorm $\mu_n$ on $W^{1,1}_{\mathrm{loc}}(\R_+,E)$ by $\mu_n(x(\cdot))=\int_0^n(\| x(t) \|+\| \dot{x}(t) \|)dt$. Since $\{ \mu_n \}_{n\in \N}$ is a countable separating family of seminorms, $W^{1,1}_{\mathrm{loc}}(\R_+,E)$ is a Fr\'echet space under the compatible metric $d$ given by
$$
d(x(\cdot),y(\cdot))=\max_{n\in \N}\frac{\mu_n(x(\cdot)-y(\cdot))}{2^n(1+\mu_n(x(\cdot)-y(\cdot)))}, \quad x(\cdot),y(\cdot)\in W^{1,1}_{\mathrm{loc}}(\R_+,E).  
$$
An element in $W^{1,1}_{\mathrm{loc}}(\R_+,E)$ is called an \textit{arc}. When $\R_+$ is replaced by a compact interval $I$ of $\R_+$, the above definition simply leads to that of the Sobolev space $W^{1,1}(I,E)$ normed by $\| x(\cdot) \|_{1,1}=\int_I(\| x(t) \|+\| \dot{x}(t) \|)dt$.  

Let $L:\R_+\times E\times E\to \R\cup\{ +\infty \}$ be an integrand. Given an arc $x(\cdot)\in W^{1,1}_{\mathrm{loc}}(\R_+,E)$, the improper integral is defined by 
$$
\int_t^\infty L(s,x(s),\dot{x}(s))ds=\lim_{T\to \infty}\int_t^T L(s,x(s),\dot{x}(s))ds
$$
for every $t\in \R_+$ provided that the above limit does exist. Let $\Gamma:\R_+\times E\rightsquigarrow E$ be a multifunction. The variational problem under investigation is to minimize the improper integral functional over the feasibility constraint governed by the differential inclusion:
\begin{equation}
\label{P}
\begin{aligned}
& \inf_{x(\cdot)\in W^{1,1}_{\mathrm{loc}}([t,\infty),E)}\int_t^\infty L(s,x(s),\dot{x}(s))ds \\
& \text{s.t.}\ \dot{x}(s)\in \Gamma(s,x(s)) \ \text{a.e.\ $s\in [t,\infty)$}, \ x(t)=\xi.
\end{aligned}
\tag{$\mathrm{P}_t$}
\end{equation}
An arc satisfying the above differential inclusion is called an \textit{admissible trajectory}. Define the set of admissible trajectories starting at time $t\in \R_+$ from a given initial condition $\xi\in E$ by
$$
\A_{(t,\xi)}:=\left\{ x(\cdot)\in W^{1,1}_{\mathrm{loc}}([t,\infty),E) \mid \dot{x}(s)\in \Gamma(s,x(s))\,\,\text{a.e.\,$s\in [t,\infty)$},\,x(t)=\xi \right\}.
$$
Then the \textit{value function} $V:\R_+\times E\to \R\cup\{ \pm\infty \}$ is defined by
$$
V(t,\xi):=\inf_{x(\cdot)\in \A_{(t,\xi)}}\int_t^\infty L(s,x(s),\dot{x}(s))ds. 
$$
Here, we set $\inf \emptyset=+\infty$ if $\A_{(t,\xi)}$ is empty or if for every $x(\cdot)\in \A_{(t,\xi)}$ the integral $\int_t^\infty L(s,x(s),\dot{x}(s))ds$ is not well-defined. The \textit{effective domain} of $V$ is given by $\mathrm{dom}\,V=\{ (t,x)\in \R_+\times E \mid V(t,x)<+\infty \}$; $V$ is said to be \textit{proper} if $\mathrm{dom}\,V$ is nonempty and $V(t,\xi)>-\infty$ for every $(t,\xi)\in \R_+\times E$. For every $(t,\xi)\in \mathrm{dom}\,V$, an admissible trajectory $x(\cdot)\in \A_{(t,\xi)}$ is said to be \textit{optimal} for \eqref{P} if it satisfies $\int_t^\infty L(s,x(s),\dot{x}(s))ds=V(t,\xi)>-\infty$. For given $x\in E$, the multifunction $\Gamma(\cdot,x)\rightsquigarrow E$ is said to be \textit{measurable} if the set $\{ t\in \R_+\mid \Gamma(t,x)\cap O\ne \emptyset \}$ belongs to $\L$ for every open subset $O$ of $E$. 

The standing hypothesis are described as follows. 
\begin{description}
\item[$\mathbf{(H_1)}$] $\A_{(t,\xi)}$ is nonempty for every $(t,\xi)\in \R_+\times E$. 
\item[$\mathbf{(H_2)}$] $L(\cdot,x,y)$ is measurable for every $(x,y)\in E\times E$. 
\item[$\mathbf{(H_3)}$] There exist an integrable function $l_1:\R_+\to \R_+$ and a locally bounded, integrable function $l_2:\R_+\to \R_+$ such that 
$$
|L(t,0,0)|\le l_1(t)
$$
and  
$$
|L(t,x,y)-L(t,x',y')|\le l_1(t)\| x-x' \|+l_2(t)\| y- y' \|
$$ 
for every $t\in \R_+$ and every $(x,y), (x',y')\in E\times E$.
\item[$\mathbf{(H_4)}$] $\Gamma$ has nonempty closed values.
\item[$\mathbf{(H_5)}$] $\Gamma(\cdot,x)$ is measurable for every $x\in E$.
\item[$\mathbf{(H_6)}$] There exist a locally integrable function $\gamma:\R_+\to \R_+$ such that 
$$
\Gamma(t,0)\subset \gamma(t)B
$$
and 
$$
\Gamma(t,x)\subset \Gamma(t,x')+\gamma(t)\| x-x' \|B
$$ 
for every $t\in \R_+$ and $x,x'\in E$.
\item[$\mathbf{(H_7)}$] The Lipschitz modulus functions satisfy the integrability conditions: 
\begin{eqnarray*}
\int_0^\infty\left[ \exp\left (\int_0^s\gamma(\tau)d\tau \right)\left( 1+\int_0^s\gamma(\tau)d\tau \right)\left( l_1(s)+l_2(s)\gamma(s) \right) \right]ds<\infty.
\end{eqnarray*}
\end{description}
Since the integrand $L$ is assumed to be a Carath\'{e}odory function in $\mathrm{(H_2)}$ and $\mathrm{(H_3)}$, it is jointly measurable on $\R_+\times E\times E$ with respect to the product $\sigma$-algebra $\L\otimes\mathrm{Borel}(E,\| \cdot \|)\otimes\mathrm{Borel}(E,\| \cdot \|)$ whenever $E$ is a separable Banach space; see \citet[Lemma 8.2.6]{af90}. Hence, $L$ is a normal integrand; see Appendix \ref{subsec2} for the definition. Hypothesis $\mathrm{(H_7)}$ guarantees the integrable boundedness of $\{ L(\cdot,x(\cdot),\dot{x}(\cdot))\mid x(\cdot)\in \A_{(t,\xi)} \}$ over the interval $[t,\infty)$ for every $(t,\xi)\in \R_+\times E$ and is needed to prove Theorem \ref{thm2} below.

In the rest of the paper, $E$ is assumed to be separable. 

\begin{thm}
\label{thm2}
If $\mathrm{(H_1)}$--$\mathrm{(H_7)}$ hold, then $V$ is bounded and lower semicontinuous on $\R_+\times E$, and $V(t,\,\cdot\,)$ is Lipschitz of rank $k(t)$ on $E$ for every $t\in \R_+$ with a continuous decreasing function $k:\R_+\to \R_+$ satisfying $k(t)\to 0$ as $t\to \infty$. 
\end{thm}

\noindent
The proof is deferred to Subsection \ref{apdx3}.

\begin{rem}
Since we impose the conditions which guarantee the integrability of the integrand on the set of admissible trajectories, the optimality criterion is unambiguous. When the integrability condition over the infinite horizon fails, there are several optimality criteria; see \citet{chl91,ss87,za06}. For the derivation of the necessary condition under (weak) overtaking optimality in the finite-dimensional setting, see \citet{ha74,ta82,ta84}. 
\end{rem}

\subsection{Subdifferentials of the Value Function}
In the following we always assume that optimal trajectories for $(\mathrm{P}_0)$ exist. To obtain an existence result in our framework, one needs standard convexity hypotheses. For the case with finite-dimensional control systems with an infinite horizon, see \citet{cf18}. For the case with reflexive, separable Banach space valued semilinear control systems with a finite horizon, see \citet{cf92}.  

Let us denote by $L^+_x(t,\bar{x},\bar{y};v)$ the upper partial directional derivative of $L(t,\cdot,\bar{y})$ at $\bar{x}\in E$ in the direction $v\in E$; $L^+_y(t,\bar{x},\bar{y};v)$ has an obvious meaning. Then $\partial^+_xL(t,\bar{x},\bar{y})$ is the Dini--Hadamard partial superdifferential of $L(t,\cdot,\bar{y})$ at $\bar{x}$; $\partial^+_yL(t,\bar{x},\bar{y})$ has a similar meaning. The Clarke partial directional derivatives $L^\circ_x(t,\bar{x},\bar{y};v)$ and  $L^\circ_y(t,\bar{x},\bar{y};v)$, and the Clarke partial subdifferentials $\partial^\circ_xL(t,\bar{x},\bar{y})$ and $\partial^\circ_yL(t,\bar{x},\bar{y})$ are defined in a similar way. Recall that $d_{\Gamma(t,x)}:E\to \R$ is the distance function from the set $\Gamma(t,x)$ and denote by $N_{\Gamma(t,x)}(y)\subset E^*$ the Clarke normal cone to $\Gamma(t,x)$ at $y\in \Gamma(t,x)$.  

We need another continuity assumption on $\Gamma$ that replaces $\mathrm{(H}_5)$: 

\begin{description}
\item[$\mathbf{(H_5')}$] $\Gamma(\cdot,x)$ is lower semicontinuous for every $x\in E$.  
\end{description}

Our results below concern the subdifferentiability of the value function. We neither impose any convexity assumptions, nor request the interiority of the optimal trajectory. This  improves results from \citet{bs79,bs82,blv96,rzs12,ta84}. 

\begin{thm}
\label{thm3}
Let $x_0(\cdot)\in \A_{(0,\xi)}$ be an optimal trajectory for $(\mathrm{P}_0)$. If $\mathrm{(H_1)}$--$\mathrm{(H}_4)$, $\mathrm{(H}_5')$, $\mathrm{(H}_6)$, and $\mathrm{(H}_7)$ hold, then:
\begin{enumerate}[\rm (i)]
\item $V^-_x(t,x_0(t);\dot{x}_0(t)-v)\le L(t,x_0(t),v)-L(t,x_0(t),\dot{x}_0(t))$ a.e.\ $t\in \R_+$ for every $v\in \Gamma(t,x_0(t))$; 
\item $-\partial^-_xV(t,x_0(t))\subset \partial^\circ_yL(t,x_0(t),\dot{x}_0(t))+N_{\Gamma(t,x_0(t))}(\dot{x}_0(t))$ a.e.\ $t\in \R_+$.
\end{enumerate}
Moreover, if $L(t,x_0(t),\cdot)$ is strictly differentiable at $\dot{x}_0(t)$, then: 
$$
-\partial^-_xV(t,x_0(t))\subset \nabla_yL(t,x_0(t),\dot{x}_0(t))+N_{\Gamma(t,x_0(t))}(\dot{x}_0(t)) \ \text{a.e.\ $t\in \R_+$}.
$$
Furthermore, if $V(t,\cdot)$ is regular at $x_0(t)$ and $\partial^-_xV(t,x_0(t))$ is a singleton, then $V(t,\cdot)$ is Gateaux differentiable at $x_0(t)$ with:
$$
-\nabla_xV(t,x_0(t))=\nabla_yL(t,x_0(t),\dot{x}_0(t))+q(t) \quad \text{a.e.\ $t\in \R_+$},
$$
where $q:\R_+\to E^*$ is a Borel measurable selector from $N_{\Gamma(\cdot,x_0(\cdot))}(\dot{x}_0(\cdot)):\R_+\rightsquigarrow E^*$ with respect to the weak$^*\!$ topology of $E^*$.  
\end{thm}

\noindent
For the proof, see Subsection \ref{apdx4}. 

\begin{exmp}
\label{exmp5}
Let $g_i:\R_+\times E\times E\to \R$, $i=1,2,\dots,m$, be Carath\'{e}odory functions and define the velocity multifunction $\Gamma:\R_+\times E\rightsquigarrow E$ by
$$
\Gamma(t,x):=\left\{ y\in E\mid g_i(t,x,y)\le 0,\,i=1,2,\dots,m \right\}. 
$$
Let $x_0(\cdot)\in W^{1,1}_{\mathrm{loc}}(\R_+,E)$ and
$$
I(t):=\{ i\in \{ 1,2,\dots,m \}\mid g_i(t,x_0(t),\dot{x}_0(t))=0 \}
$$
be the active constraint indices at $(t,x_0(t),\dot{x}_0(t))\in \R_+\times E\times E$. Assume that $g_i(t,x_0(t),\cdot)$ has the strict derivative at $\dot{x}_0(t)$ for each $i\in I(t)$ and the constraint qualification $0\not\in \mathrm{co}\{ \nabla_yg_i(t,x_0(t),\dot{x}_0(t))\mid i\in I(t) \}$ holds. It follows from Proposition \ref{exmp6} that: 
$$
N_{\Gamma(t,x_0(t))}(\dot{x}_0(t))=\left\{ \sum_{i\in I(t)}\lambda_i\nabla_yg_i(t,x_0(t),\dot{x}_0(t))\in E^* \mid \lambda_i\ge 0 \ \forall i\in I(t) \right\}. 
$$
Under the hypotheses of Theorem \ref{thm3}, we have
$$
-\nabla_xV(t,x_0(t))=\nabla_yL(t,x_0(t),\dot{x}_0(t))+\sum_{i\in I(t)}\lambda_i(t)\nabla_yg_i(t,x_0(t),\dot{x}_0(t))
$$
for some $\lambda_i(t)\ge 0$ with $i\in I(t)$. By the measurable selection theorem, the mapping $t\mapsto \lambda_i(t)$ can be chosen in a measurable way. Under the convexity hypothesis with the constraint qualification, \citet{rzs12} provided sufficient conditions for the differentiability of the value function with the finite-dimensional state constraint without the interiority conditions $\mathrm{(H_8)}$ below. 
\end{exmp}

\section{Euler--Lagrange Conditions and the Maximum Principle}
\subsection{Necessary Conditions under the Interiority Assumption} 
A function $p:\R_+\to E^*$ is said to be \textit{locally absolutely continuous} if its restriction to the bounded closed interval $[0,\tau]$ is absolutely continuous for every $\tau>0$, i.e., for every $\tau>0$ and $\varepsilon>0$ there exists $\delta>0$ such that $0\le t_1<\tau_1\le t_2<\tau_2<\dots \le t_n<\tau_n\le \tau$ and $\sum_{i=1}^n| t_i-\tau_i|<\delta$ imply $\sum_{i=1}^n\| p(t_i)-p(\tau_i) \|<\varepsilon$. A function $p(\cdot)$ is said to be \textit{weakly$^*\!$ differentiable} at $t>0$ if there exists $x^*\in E^*$ such that   
$$
\lim_{h\to 0}\left\langle \frac{p(t+h)-p(t)}{h},x \right\rangle=\langle x^*,x \rangle \quad\text{for every $x\in E$}. 
$$
Then vector $x^*$ is called the \textit{weak$^*\!$ derivative} of $p$ at $t$ and is denoted by $\dot{p}(t)$ with $d\langle p(t),x \rangle/dt=\langle \dot{p}(t),x \rangle$ for every $x\in E$. 

We impose a feasibility assumption on the perturbation around a specific optimal trajectory $x_0(\cdot)\in \A_{(0,\xi)}$ as follows.

\begin{description}
\item[$\mathbf{(H_8)}$] For every $T>0$ there exists $\eta>0$ such that 
$$
(x_0(t)+\eta B,\dot{x}_0(t))\subset \mathrm{gph}\,\Gamma(t,\cdot) \quad\text{a.e.\ $t\in [0,T]$}.
$$ 
\end{description}
This is equivalent to saying that for every $T>0$ there exists $\eta>0$ such that a.e.\ $t\in [0,T]$ we have $\dot{x}_0(t)\in \Gamma(t,x_0(t)+\eta v)$ for every $v\in B$. $\mathrm{(H_8)}$ is a weaker condition than the interiority condition  imposed in \citet{bs82}: 

\begin{description}
\item[$\mathbf{(H_8')}$] There exists $\eta>0$ such that 
$$
\left( x_0(t)+\eta B,\dot{x}_0(t) \right)\subset \mathrm{gph}\,\Gamma(t,\cdot) \quad\text{a.e.\ $t\in \R_+$}.
$$ 
\end{description}
Furthermore, $\mathrm{(H_8)}$ is a partial improvement of the interiority condition imposed in \citet{bs79}: 

\begin{description}
\item[$\mathbf{(H_8'')}$] There exist $T>0$ and $\eta>0$ such that 
$$
\left( x_0(t)+\eta B,\dot{x}_0(t)+\eta B \right)\subset \mathrm{gph}\,\Gamma(t,\cdot) \quad\text{a.e.\ $t\in [0,T]$}.
$$ 
\end{description}

Define the \textit{Hamiltonian} $H:\R_+\times E\times E^*\to \R\cup\{ +\infty \}$ by
$$
H(t,x,x^*):=\sup_{y\in \Gamma(t,x)}\left\{ \langle x^*,y \rangle-L(t,x,y) \right\}. 
$$

Now we are ready to present an extension of the Euler--Lagrange necessary condition and the maximum principle with the transversality condition at infinity. In the theorem below we use the notions of weak$^*\!$ scalar measurability and Gelfand integrals whose definitions are recalled in Appendix \ref{subsec1}.  

\begin{thm}
\label{thm4}
Suppose that $\mathrm{(H_1)}$--$\mathrm{(H}_4)$, $\mathrm{(H}_5')$, $\mathrm{(H}_6)$, $\mathrm{(H}_7)$, and $\mathrm{(H_8)}$ hold. If $\partial^-_xV(0,x_0(0))$ is nonempty and $\partial^+_x L(t,x_0(t),\dot{x}_0(t))$ is nonempty a.e.\ $t\in\R_+$, then for every $x^*\in \partial^-_xV(0,x_0(0))$ and weakly$^*\!$ scalarly measurable selector $f:\R_+\to E^*$ from the Dini--Hadamard superdifferential mapping $\partial^+_xL(\cdot,x_0(\cdot),\dot{x}_0(\cdot)):\R_+\rightsquigarrow E^*$, the locally absolutely continuous function $p:\R_+\to E^*$ defined by $p(t):=\int_0^tf(s)ds-x^*$ as a Gelfand integral satisfies:
\begin{enumerate}[\rm(i)]
\item $-p(t)\in \partial^-_xV(t,x_0(t))$ for every $t\in\R_+$; 
\item $p(t)\in \partial^\circ_yL(t,x_0(t),\dot{x}_0(t))+N_{\Gamma(t,x_0(t))}(\dot{x}_0(t))$ a.e.\ $t\in \R_+$;  
\item $\dot{p}(t)\in \partial^+_x L(t,x_0(t),\dot{x}_0(t))$ a.e.\ $t\in\R_+$; 
\item $H(t,x_0(t),p(t))=\langle p(t),\dot{x}_0(t) \rangle-L(t,x_0(t)),\dot{x}_0(t))$ a.e.\ $t\in \R_+$;  
\item $\displaystyle\lim_{t\to \infty}p(t)=0$,
\end{enumerate}
where $\dot{p}(t)$ denotes the weak$^*\!$ derivative of $p(\cdot)$ at $t\in \R_+$. In particular, if $\partial^-_xV(0,x_0(0))$ is nonempty, then $\partial^-_xV(t,x_0(t))$ is nonempty for every $t\in \R_+$. 
\end{thm}

\noindent
For the proof, see Subsection \ref{apdx2}. 

If $\mathrm{gph}\,\Gamma(t,\cdot)$ is convex on which $L(t,\cdot,\cdot)$ is convex for every $t\in \R_+$, then $H(t,\cdot,x^*)$ is concave on $E$ for every $x^*\in E^*$ and $H(t,x,\cdot)$ is convex on $E^*$ for every $x\in E$. It thus follows from the Fenchel duality that the Euler--Lagrange conditions (ii) and (iii) of Theorem \ref{thm4} are equivalent to the familiar Hamiltonian conditions $\dot{x}_0(t)\in \partial_pH(t,x_0(t),p(t))$ and $-\dot{p}(t)\in \partial_xH(t,x_0(t),p(t))$ respectively; see \citet[Theorem 6]{ro70}.

\begin{rem}
\label{rem1}
Theorem \ref{thm4} yields that ``singular'' points propagate forward along optimal trajectories, i.e., if the Dini--Hadamard subdifferential $\partial^-_xV(0,x_0(0))\ne \emptyset$ is not a singleton, then so does $\partial^-_xV(t,x_0(t))$ for every $t\in \R_+$. In the finite-dimensional control systems, this observation is done also in \citet{ta82} for fully convex variational problems and in \citet{cf18} for optimal control problems whose set of velocities are convex. Note that the nonemptiness of $\partial^-_xV(0,x_0(0))$ is an innocuous assumption because the set of points at which $\partial^-_xV(0,\cdot)$ is Dini--Hadamard subdifferentiable is dense in the separable Banach space $E$ by Theorem \ref{thm0}. 
\end{rem}

\begin{rem}
\label{rem2}
Even if the Dini--Hadamard subdifferential $\partial^-_xV(t,x_0(t))$ is a singleton, the strict derivative $\nabla_xV(t,x_0(t))$ might not exist because of the lack of convexity of the lower directional derivative $v\mapsto V^-_x(t,x_0(t);v)$. This observation makes a sharp contrast with the case where the Clarke subdifferential $\partial^\circ_xV(t,x_0(t))$ is a singleton, in which case $\partial^\circ_xV(t,x_0(t))$ coincides with $\nabla_xV(t,x_0(t))$; see for detail the proof of \citet[Proposition 2.2.4]{cl83}.
\end{rem}

\subsection{Necessary Conditions without the Interiority Assumption}
Hypothesis $\mathrm{(H_8)}$ in Theorem \ref{thm4} is stringent, mostly because the velocity multifunction $\Gamma$ is too general. Hence, a ``structural assumption'' on the optimal trajectory $x_0(\cdot)$ compensates this generality. If instead some ``structural'' assumptions are imposed on $\Gamma$, then hypothesis $\mathrm{(H_8)}$ can be omitted. To illustrate  this observation, we consider standard optimal control problems. 

Let $X$ be a complete separable metric space, $f:\R_+\times E\times X\to E$ be a velocity function, and $U:\R_+\rightsquigarrow X$ be a control multifunction. Denote by $\M(\R_+,X)$ the space of measurable functions on $\R_+$ with values in $X$. Define the integrand $\tilde{L}:\R_+\times E\times X \to \R$ by $\tilde{L}(t,x,u):=L(t,x,f(t,x,u))$.  

The optimal control problem under consideration is as follows:  

\begin{equation}
\label{oc}
\begin{aligned}
& \inf_{\begin{subarray}{l}x(\cdot)\in W^{1,1}_\mathrm{loc}(\R_+,E) \\ u(\cdot)\in \M(\R_+,X) \end{subarray}}\int_0^\infty \tilde{L}(t,x(t),u(t))dt \\
& \text{s.t. } u(t)\in U(t) \text{ a.e.\ $t\in \R_+$}, \\
& \hspace{0.70cm} \dot{x}(t)=f(t,x(t),u(t)) \ \text{a.e.\ $t\in \R_+$}, \quad x(0)=\xi. 
\end{aligned}
\tag{$\tilde{\mathrm{P}}_0$}
\end{equation}
The Hamiltonian for problem \eqref{oc} is given by:
$$
H(t,x,x^*)=\sup_{u\in U(t)}\left\{ \langle x^*,f(t,x,u) \rangle-\tilde{L}(t,x,u) \right\}. 
$$
The velocity multifunction is defined by $\Gamma(t,x):=f(t,x,U(t))$. Now impose ``usual'' assumptions on $f$ and $U$ in order that $\Gamma$ satisfies $\mathrm{(H_1)}$--$\mathrm{(H}_4)$, $\mathrm{(H}_5')$, $\mathrm{(H}_6)$, and $\mathrm{(H}_7)$. Note that $\partial^+_x\tilde{L}(t,x,u)$ is nonempty at $(t,x,u)\in \R_+\times E\times X$ whenever so is $\partial^+_{x,y}L(t,x,f(t,x,u))$ and $f(t,\cdot,u)$ is Gateaux differentiable at $x$. 

Denote by $\L(E)$ the space of bounded linear operators on $E$. 

The following reasonable hypothesis is a ``structural assumption'' on $f$ that dispenses with $\mathrm{(H_8)}$. 

\begin{description}
\item[$\mathbf{(H_9)}$] 
\begin{enumerate}[(i)]
\item $f$ is a Carath\'{e}odory function, i.e., $f(\cdot,x,u)$ is measurable for every $(x,u)\in E\times X$ and $f(t,\cdot,\cdot)$ is continuous for every $t\in \R_+$. 
\item For every $R>0$ and $T>0$ there exists an integrable function $k:[0,T]\to \R$ such that: 
\begin{enumerate}
\item $\| f(t,x,u) \|\le k(t)$ for every $t\in [0,T]$, $x\in RB$, and $u\in U(t)$;
\item $f(t,\cdot,u)$ is Lipschitz of rank $k(t)$ on $RB$ for every $t\in [0,T]$ and $u\in U(t)$.
\end{enumerate}
\item $f(t,\cdot,u)$ is Fr\'{e}chet differentiable on $E$ for every $(t,u)\in \R_+\times X$ and the mapping $(t,x,u)\mapsto \nabla_xf(t,x,u)$ is continuous in the uniform operator topology of $\L(E)$. 
\end{enumerate}
\end{description} 
Conditions $\mathrm{(H_9)}$-(i), (ii) guarantee the existence of solutions of the integral equation 
$$
x(t)=\int_0^tf(s,x(s),u(s))ds+\xi \quad \text{for every $t\in \R_+$}
$$ 
for any control $u(\cdot)\in \M(\R_+,X)$, where the locally absolutely continuous function $x(\cdot):\R_+\to E$ is a unique mild solution to the ordinary differential equation (ODE) in \eqref{oc} (see \citet[Theorem 5.5.1]{fa99}), which has the strong derivative $\dot{x}(t)$ a.e.\ $t\in \R_+$ in view of the separability of $E$ and the Lebesgue differentiation theorem. 

Let $(x_0(\cdot),u_0(\cdot))\in W^{1,1}_\mathrm{loc}(\R_+,E)\times \M(\R_+,X)$ be an optimal trajectory-control pair for optimal control problem \eqref{oc}. Denote by $\nabla_xf(s,x_0(s),u_0(s))^*$ in $\L(E^*)$ the adjoint operator of $\nabla_xf(s,x_0(s),u_0(s))$ in $\L(E)$. 

\begin{thm}
\label{thm5}
Suppose that $\mathrm{(H_1)}$--$\mathrm{(H}_4)$, $\mathrm{(H}_5')$, $\mathrm{(H}_6)$, $\mathrm{(H}_7)$, and $\mathrm{(H}_9)$ hold with $\Gamma(t,x)=f(t,x,U(t))$. If $\partial^-_xV(0,x_0(0))$ is nonempty and the Dini--Hadamard superdifferential mapping $\partial^+_x\tilde{L}(\cdot,x_0(\cdot),u_0(\cdot)):\R_+\rightsquigarrow E^*$ admits a locally Bochner integrable selector, then there exists a locally absolutely continuous function $p:\R_+\to E^*$ such that:
\begin{enumerate}[\rm(i)]
\item $-p(t)\in \partial^-_xV(t,x_0(t))$ for every $t\in\R_+$; 
\item $p(t)\in \partial^\circ_yL(t,x_0(t),f(t,x_0(t),u_0(t)))+N_{\Gamma(t,x_0(t))}(f(t,x_0(t),u_0(t)))$ a.e.\ $t\in \R_+$;  
\item $-\dot{p}(t)\in \nabla_xf(t,x_0(t),u_0(t))^*p(t)-\partial^+_x \tilde{L}(t,x_0(t),u_0(t))$ a.e.\ $t\in \R_+$;
\item $H(t,x_0(t),p(t))=\langle p(t),f(t,x_0(t),u_0(t)) \rangle-\tilde{L}(t,x_0(t),u_0(t))$ a.e.\ $t\in \R_+$;  
\item $\displaystyle\lim_{t\to \infty}p(t)=0$, 
\end{enumerate}
where $\dot{p}(t)$ denotes the strong derivative of $p(\cdot)$ at $t\in \R_+$. In particular, if $\partial^-_xV(0,x_0(0))$ is nonempty, then $\partial^-_xV(t,x_0(t))$ is nonempty for every $t\in \R_+$. 
\end{thm}

\noindent
The proof is provided in Subsection \ref{apdx2}.

\begin{rem}
The existence of locally Bochner integrable selectors from the Dini--Hadamard superdifferential mapping $t\rightsquigarrow \partial^+_x\tilde{L}(t,x_0(t),u_0(t))$ follows from $\mathrm{(H_3)}$ and $\mathrm{(H_9)}$ whenever $E^*$ is separable in the dual norm. For the case with nonseparable $E^*$, the Fr\'{e}chet differentiability of the integrand $L(t,\cdot,\cdot)$ on $E\times E$ and the continuity of $(t,x,y)\mapsto (\nabla_xL(t,x,y),\nabla_yL(t,x,y))$ in the dual norm of $E^*\times E^*$ guarantee the local Bochner integrability of $t\mapsto \nabla_x\tilde{L}(t,x_0(t),u_0(t))$ in $E^*$ under $\mathrm{(H_3)}$ and $\mathrm{(H_9)}$. If $\partial^+_{x,y}L(t,x_0(t),\dot{x}_0(t))$ is nonempty, then take any $(p,q)\in \partial^+_{x,y}L(t,x_0(t),\dot{x}_0(t))$ and observe that for every $v\in E$ and $u\in U(t)$, we have:
\begin{align*}
\tilde{L}^+_x(t,x,u;v)
& \le L^+_{x,y}(t,x,f(t,x,u);v,\nabla_xf(t,x,u)v) \\
& \le \langle p,v\rangle+\langle q,\nabla_xf(t,x,u)v \rangle=\langle p+\nabla_xf(t,x,u)^*q,v \rangle.
\end{align*}
Thus, instead of using the function $g:\R_+\to E^*$ in the proof of Theorem \ref{thm5} below, we could use as well any locally Bochner integrable selector $(\alpha(t),\beta(t))\in \partial^+_{x,y}L(t,x_0(t),\dot{x}_0(t))$ and write the adjoint equation involving $g(t)=\alpha(t)+\nabla_xf(t,x_0(t),u_0(t))^*\beta(t)$.  
\end{rem}

\section{An Application: Spatial Ramsey Growth Models}
\subsection{Ramsey Meets Hotelling}
Consider the spatial Ramsey growth model with a general reduced form explored in \citet{bcf13,bcz09,bxy14,br04,czb08} in the specific form. Let $I=[0,1]$ be the unit interval such that the endpoints $0,1\in I$ are identified. Then $I$ is homeomorphic to the unit circle in which a spatial parameter $\theta\in I$ is a location of agents along the lines of \citet{ho29}. Let $W:\R_+\times \R_+\times I\to \R$ be a function such that $W(\cdot,\cdot,\theta)$ is an instantaneous utility function at location $\theta\in I$ satisfying $W(a,b,0)=W(a,b,1)$ for every $(a,b)\in \R_+\times \R_+$, where $a$ denotes a capital stock and $b$ a net investment. Let $F:\R_+\times \R_+\times I\to \R_+$ be a function such that $F(\cdot,\cdot,\theta)$ is a production function at $\theta$ satisfying $F(a,c,0)=F(a,c,1)$ for every $(a,c)\in \R_+\times \R_+$, where $c$ denotes a consumption and output $F(a,c,\theta)$ is a net investment at $\theta$ for every $(a,c)$. Let $r>0$ be a discount rate. For simplicity, we assume no depreciation of capital stock.  

Let $x:\R_+\times I\to \R_+$ be a capital stock trajectory in which $x(t,\theta)$ is a current capital stock, $\partial x(t,\theta)/\partial t$ is a current capital accumulation, $u:\R_+\times I\to \R_+$ denotes a consumption trajectory in which $u(t,\theta)$ is a current consumption, and $\tau:\R_+\times I\to \R$ denotes a net transfer trajectory in which $\tau(t,\theta)$ is a current net transfer, respectively at period $t\in \R_+$ and location $\theta\in I$. The capital accumulation process is described by 
$$
\frac{\partial x(t,\theta)}{\partial t}=F(x(t,\theta),u(t,\theta),\theta)+\tau(t,\theta)
$$ 
for a.e.\ $t\in \R_+$ and for every $\theta\in I$. If $\tau(t,\theta)\equiv 0$, then the model describes an autarkic economy in which no capital movement occurs among locations; see \citet{br04}. For the case where $\tau(t,\theta)$ is a parabolic term, see \citet{bcf13,bcz09,br04,czb08}. The choice of function spaces depends upon the specification of a transfer term. Following the forementioned works, we focus here on the autarkic case in which capital stock and consumption change smoothly in locations. 

Let $C^2(I)$ be the space of twice continuously differentiable functions on $I$ with their values equal at the end points $\theta=0,1$, endowed with the $C^2$-norm $\| x \|_{C^2(I)}:=\sup_{\theta\in I}\{ |x(\theta)|+|x'(\theta)|+|x''(\theta)| \}$, which makes $C^2(I)$ a separable Banach space. Denote by $C^2_+(I)$ the positive cone of $C^2(I)$ consisting of all nonnegative functions in $C^2(I)$. The problem under investigation is:
\begin{equation}
\label{Q}
\begin{aligned}
& \max\int_0^\infty \int_I e^{-rt}W\left( x(t,\theta),\frac{\partial x(t,\theta)}{\partial t},\theta \right)d\theta dt  \\
& \text{\,s.t. } \frac{\partial x(t,\theta)}{\partial t}=F(x(t,\theta),u(t,\theta),\theta), \quad u(t)\in U(t) \\
& \hspace{4.5cm} \text{ a.e.\ $t\in \R_+$ for every $\theta\in I$}, \\
& \hspace{0.8cm} x(t,0)=x(t,1), \ u(t,0)=u(t,1) \ \text{a.e.\ $t\in \R_+$}, \\
& \hspace{0.8cm} x(0,\theta)=\xi(\theta) \ \text{for every $\theta\in I$}.
\end{aligned}
\tag{$\mathrm{Q}_0$}
\end{equation}
Here, the control set $U(t)$ is a subset of $C^2_+(I)$ for every $t\in \R_+$. The maximization is taken over all \textit{nonnegative} functions $x(\cdot,\cdot)$ in the function space such that $x(\cdot,\theta)$ is a.e.\ differentiable on $\R_+$ for every $\theta \in I$ with $x(t,\cdot)$ and $\partial x(t,\cdot)/\partial t$ belonging to $C^2(I)$ a.e.\ $t\in \R_+$, and over all functions $u(\cdot,\cdot)$ such that $u(t,\cdot)\in U(t)$ and $u(\cdot,\theta)$ is measurable on $\R_+$ for every $\theta\in I$ satisfying the parametrized ODE above, where the initial condition at location $\theta$ is given by $\xi(\theta)$ with $\xi\in C^2_+(I)$. 

Throughout this section, we assume the following.  

\begin{assmp}
\label{assmp}
\begin{enumerate}[(i)]
\item $W$ has an extension to $\R\times \R\times I$ (which we do not relabel) such that $W(a,b,\cdot)$ is measurable on $I$ for every $(a,b)\in \R\times \R$, $W(\cdot,\cdot,\theta)$ is continuously differentiable on $\R\times \R$ for every $\theta\in I$, and its partial derivatives are bounded uniformly in $(a,b,\theta)\in \R\times \R \times I$. 
\item $F$ has a thrice continuously differentiable extension to $\R\times \R\times I$ (which we do not relabel) such that every partial derivative of any order less or equal $3$ is bounded uniformly in $(a,c,\theta)\in \R\times \R \times I$.
\item There exists a bounded closed subset $X$ of $C^2_+(I)$ such that $U(t)\subset X$ for every $t\in \R_+$. 
\end{enumerate}
\end{assmp}

Define the integrands $L:\R_+\times C^2(I)\times C^2(I)\to \R$ by
$$
L(t,x(\cdot),y(\cdot)):=-e^{-rt}\int_IW(x(\theta),y(\theta),\theta)d\theta
$$
and $\tilde{L}:\R_+\times C^2(I)\times X\to \R$ by
$$
\tilde{L}(t,x(\cdot),u(\cdot)):=-e^{-rt}\int_IW(x(\theta),F(x(\theta),u(\theta),\theta),\theta)d\theta
$$
respectively. Consider the velocity function $f:C^2(I)\times X\to C^2(I)$ defined by 
$$
f(x(\cdot),u(\cdot)):=F(x(\cdot),u(\cdot),\cdot). 
$$
Here, $f(\cdot,u(\cdot))$ is Fr\`{e}chet differentiable on $C^2(I)$ and its Fr\`{e}chet derivative $\nabla_xf(x(\cdot),u(\cdot))\in \L(C^2(I))$ can be calculated as
$$
\nabla_xf(x(\cdot),u(\cdot))v(\cdot)=\frac{\partial F(x(\cdot),u(\cdot),\cdot)}{\partial a}v(\cdot)
$$ 
for every $v(\cdot)\in C^2(I)$. By construction, it is evident that 
$$
\tilde{L}(t,x(\cdot),u(\cdot))=L(t,x(\cdot),f(x(\cdot),u(\cdot)))
$$ 
for every $(x(\cdot),u(\cdot))\in C^2(I)\times X$. Define the velocity multifunction $\Gamma:\R_+\times C^2(I)\rightsquigarrow C^2(I)$ by $\Gamma(t,x(\cdot)):=f(x(\cdot),U(t))$. 

We then convert the problem \eqref{Q} into the minimization one of the form \eqref{oc} in the setting with $E=C^2(I)$ and $X\subset C^2_+(I)$. If $(x_0(\cdot),u_0(\cdot))\in W^{1,1}_\mathrm{loc}(\R_+,C^2(I))\times \M(\R_+,X)$ is an optimal trajectory-control pair of \eqref{oc}, then $(x_0(\cdot),u_0(\cdot))$ is a solution to the associated problem \eqref{Q}, and vice versa because any admissible trajectory of \eqref{oc} stays in the nonnegative orthant. It is easy to see that Assumption \ref{assmp} guarantees hypotheses in Theorem \ref{thm5}. In particular, Hypothesis $\mathrm{(H_7)}$ is satisfied whenever $r>0$ is large enough.

\subsection{Necessary Conditions for Optimality}
Let $C(I)$ be the space of continuous functions on $I$ endowed with the sup norm and $\mathit{ca}(I)$ be the space of signed Borel measures on $I$. Since each $x\in C^2(I)$ is represented by
$$
x(\theta)=x(0)+x'(0)\theta+\int_0^\theta\int_0^\sigma x''(\omega)d\omega d\sigma \quad\text{for every $\theta\in I$}
$$
with $x(0),x'(0)\in \R$ and $x''\in C(I)$, the Banach space $C^2(I)$ is identified with the direct sum $\R\oplus \R \oplus C(I)$. Hence, $C^2(I)^*=\R\oplus \R \oplus \mathit{ca}(I)$ and each $x^*\in C^2(I)^*$ has the form
$$
\langle x^*,x \rangle=\alpha_0x(0)+\alpha_1x'(0)+\int_Ix''(\theta)d\mu \quad\text{for every $x\in C^2(I)$}
$$
for some constants $\alpha_0,\alpha_1\in \R$ and a signed Borel measure $\mu\in \mathit{ca}(I)$; see \citet[Exercise IV.13.36]{ds58}. Hence, the adjoint variables in the spatial Ramsey growth model take values in $\R\oplus \R \oplus \mathit{ca}(I)$. 

Let $(x_0(\cdot),u_0(\cdot))\in W^{1,1}(\R_+,C^2(I))\times \M(\R_+,X)$ be an optimal trajecto\-ry-\hspace{0pt}control pair of \eqref{oc} and $V:\R_+\times C^2(I)\to \R$ be the value function. A direct calculation shows that for every $v\in C^2(I)$ we have
$$
\langle \nabla_xL(t,x_0(t),\dot{x}_0(t)),v \rangle=-e^{-rt}\int_I\frac{\partial W(x_0(t,\theta),\dot{x}_0(t,\theta),\theta)}{\partial a}v(\theta)d\theta.
$$
To evaluate the above integral, define $\alpha(\theta):=\partial W(x_0(t,\theta),\dot{x}_0(t,\theta),\theta)/\partial a$,  
$A(\theta):=-\int_\theta^1\alpha(\omega)d\omega$, and $B(\theta):=-\int_\theta^1 A(\omega)d\omega$. The double use of integration by parts yields  
\begin{align*}
& \int_I\alpha(\theta)v(\theta)d\theta=\left[ A(\theta)v(\theta) \right ]_0^1-\int_IA(\theta)v'(\theta)d\theta \\ {}={}
& -A(0)v(0)+B(0)v'(0)+\int_IB(\theta)v''(\theta)d\theta \\ {}={}
& \int_I\frac{\partial W(x_0(t,\theta),\dot{x}_0(t,\theta),\theta)}{\partial a}d\theta v(0) \\ 
& \qquad +\int_I\left[ \int_\theta^1\frac{\partial W(x_0(t,\theta),\dot{x}_0(t,\omega),\omega)}{\partial a}d\omega \right]d\theta v'(0) \\
& \qquad +\int_I\left[ \int_\theta^1\int_\sigma^1\frac{\partial W(x_0(t,\theta),\dot{x}_0(t,\omega),\omega)}{\partial a}d\omega d\sigma \right]v''(\theta)d\theta. 
\end{align*}
Henceforth, we obtain
$$
\langle \nabla_xL(t,x_0(t),\dot{x}_0(t)),v \rangle=a_0(t)v(0)+a_1(t)v'(0)+\int_Iv''(\theta)d\mu(t)
$$
with
\begin{align*}
& a_0(t)=-e^{-rt}\int_I\frac{\partial W(x_0(t,\theta),\dot{x}_0(t,\theta),\theta)}{\partial a}d\theta, \\
& a_1(t)=-e^{-rt}\int_I\int_\theta^1\frac{\partial W(x_0(t,\theta),\dot{x}_0(t,\omega),\omega)}{\partial a}d\omega d\theta, \\
& \frac{d\mu_0(t)}{d\theta}(\theta)=-e^{-rt}\int_\theta^1\int_\sigma^1 \frac{\partial W(x_0(t,\theta),\dot{x}_0(t,\omega),\omega)}{\partial a}d\omega d\sigma,
\end{align*}
where $\mu_0(t)\in \mathit{ca}(I)$ is given by its Radon--Nikodym derivative $d\mu_0(t)/d\theta$. Hence, $\nabla_xL(t,x_0(t),\dot{x}_0(t))=a_0(t)\oplus a_1(t)\oplus \mu_0(t)\in \R\oplus \R \oplus \mathit{ca}(I)$. Similarly, replacing $\alpha$ by $\alpha(\theta):=\partial W(x_0(t,\theta),\dot{x}_0(t,\theta),\theta)/\partial b$ in the above argument with integration by parts yields 
\begin{align*}
\langle \nabla_yL(t,x_0(t),\dot{x}_0(t)),v \rangle
& =-e^{-rt}\int_I\frac{\partial W(x_0(t,\theta),\dot{x}_0(t,\theta),\theta)}{\partial b}v(\theta)d\theta \\
& =b_0(t)v(0)+b_1(t)v'(0)+\int_Iv''(\theta)d\nu_0(t)
\end{align*}
with
\begin{align*}
& b_0(t)=-e^{-rt}\int_I\frac{\partial W(x_0(t,\theta),\dot{x}_0(t,\theta),\theta)}{\partial b}d\theta, \\
& b_1(t)=-e^{-rt}\int_I\int_\theta^1\frac{\partial W(x_0(t,\theta),\dot{x}_0(t,\omega),\omega)}{\partial b}d\omega d\theta, \\
& \frac{d\nu_0(t)}{d\theta}(\theta)=-e^{-rt}\int_\theta^1\int_\sigma^1\frac{\partial W(x_0(t,\theta),\dot{x}_0(t,\omega),\omega)}{\partial b}d\omega d\sigma. 
\end{align*}
Hence, $\nabla_yL(t,x_0(t),\dot{x}_0(t))=b_0(t)\oplus b_1(t)\oplus \nu_0(t)\in \R\oplus \R \oplus \mathit{ca}(I)$. 

If $\xi\in C^2_+(I)$ is a point such that $\partial^-_xV(0,\xi)$ is nonempty (see Remark \ref{rem1}), then in view of Theorem \ref{thm5} there exists a locally absolutely continuous function $p:\R_+\to C^2(I)^*$ such that (i) $-p(t)\in \partial^-_xV(t,x_0(t))$ for every $t\in \R_+$; (ii) $p(t)=\nabla_yL(t,x_0(t),\dot{x}_0(t))+q(t)$ a.e.\ $t\in \R_+$, where $q:\R_+\to C^2(I)^*$ is a Borel measurable selector from $t\rightsquigarrow N_{\Gamma(t,x_0(t))}(\dot{x}_0(t))$ with respect to the weak$^*\!$ topology of $C^2(I)^*$; (iii) $-\dot{p}(t)=\nabla_xf(x_0(t),u_0(t))^*p(t)-\nabla_x\tilde{L}(t,x_0(t),\dot{x}_0(t))$ a.e.\ $t\in \R_+$; (iv) $H(t,x_0(t),p(t))=\langle p(t),f(x_0(t),u_0(t)) \rangle-\tilde{L}(t,x_0(t),u_0(t))$ a.e.\ $t\in \R_+$; (v) $\lim_{t\to \infty}p(t)=0$. Since $q$ is represented as $q(t)=\beta_0(t)\oplus\beta_1(t)\oplus \nu_1(t)\in \R\oplus \R\oplus \mathit{ca}(I)$, condition (ii) can be written as $p(t)=(b_0(t)+\beta_0(t))\oplus(b_1(t)+\beta_1(t))\oplus(\nu_0(t)+\nu_1(t))$ a.e.\ $t\in \R_+$. Since 
\begin{align*}
& \nabla_x\tilde{L}(t,x_0(t),u_0(t)) \\ {}={}
& \nabla_xL(t,x_0(t),\dot{x}_0(t))+\nabla_yL(t,x_0(t),\dot{x}_0(t))\nabla_xf(x_0(t),u_0(t)) \\ {}={}
& \nabla_xL(t,x_0(t),\dot{x}_0(t))+\nabla_xf(x_0(t),u_0(t))^*(p(t)-q(t)), 
\end{align*}
condition (iii) yields $-\dot{p}(t)=-(a_0(t)\oplus a_1(t)\oplus \mu_0(t))+q(t)\nabla_xf(x_0(t),u_0(t))$, and hence, $\dot{p}(t)=(a_0(t)-\alpha_0(t))\oplus (a_1(t)-\alpha_1(t))\oplus (\mu_0(t)-\mu_1(t))$ with setting $q(t)\nabla_xf(x_0(t),u_0(t))=\alpha_0(t)\oplus\alpha_1(t)\oplus \mu_1(t)\in \R\oplus \R\oplus \mathit{ca}(I)$. 

Define the new adjoint variable by $\pi(t):=e^{rt}p(t)$ and denote it by $\pi(t)=\pi_0(t)\oplus \pi_1(t)\oplus \lambda(t)\in \R\oplus \R\oplus \mathit{ca}(I)$. The transversality condition at infinity can be written as $\lim_{t\to \infty}e^{-rt}\pi(t)=0$. It follows from condition (iii) that 
\begin{equation}
\label{eq5}
\dot{\pi}_0(t)=r\pi_0(t)-\int_I\frac{\partial W(x_0(t,\theta),\dot{x}_0(t,\theta),\theta)}{\partial a}d\theta-e^{rt}\alpha_0(t)
\end{equation}
\begin{equation}
\dot{\pi}_1(t)=r\pi_1(t)-\int_I\int_\theta^1\frac{\partial W(x_0(t,\omega),\dot{x}_0(t,\omega),\omega)}{\partial a}d\omega d\theta-e^{rt}\alpha_1(t)
\end{equation}
\begin{equation}
\begin{aligned}
\label{eq8}
\dot{\lambda}(t)(A)
& =r\lambda(t)(A)-\int_A\int_\theta^1\int_\sigma^1\frac{\partial W(x_0(t,\omega),\dot{x}_0(t,\omega),\omega)}{\partial a}d\omega d\sigma d\theta \\
& \qquad -e^{rt}\mu_1(t)(A)
\end{aligned}
\end{equation}
a.e.\ $t\in \R_+$  for every $A\in \L$. Hence, any stationary point $(\bar{x},\bar{\pi})\in C^2(I)\times C^2(I)^*$ of the dynamical system corresponding to $\dot{x}_0(t)\equiv 0$ and $\dot{\pi}(t)\equiv 0$ with the adjoint equations \eqref{eq5}--\eqref{eq8} is determined by the following conditions:  
\begin{align*}
& F(\bar{x}(\cdot),\bar{u}(\cdot),\cdot)=0, \\
& \bar{\pi}_0=r^{-1}\left( \int_I\frac{\partial W(\bar{x}(\theta),0,\theta)}{\partial a}d\theta+\alpha_0(0) \right), \\
& \bar{\pi}_1=r^{-1}\left( \int_I\int_\theta^1\frac{\partial W(\bar{x}(\omega),0,\omega)}{\partial a}d\omega d\theta+\alpha_1(0)\right), \\
& \bar{\lambda}(A)=r^{-1}\left( \int_A\int_\theta^1\int_\sigma^1\frac{\partial W(\bar{x}(\omega),0,\omega)}{\partial a}d\omega d\sigma d\theta+\mu_1(0)(A) \right) \\
& \hspace{7.5cm} \text{for every $A\in \L$}
\end{align*}
along with $(\alpha_0(t),\alpha_1(t),\mu_1(t))=(\alpha_0(0),\alpha_1(0),\mu_1(0))e^{-rt}$.

\appendix
\small{
\section{Appendix I}
\subsection{Proof of Proposition \ref{exmp6}}
\label{apdx0}
Since $T_C(\bar{x})\subset K_C(\bar{x})$, if $v\in T_C(\bar{x})$, then there exists a sequence $\{ \theta_n \}_{n\in \N}$ of positive real numbers with $\theta_n\downarrow 0$ and a sequence $\{ v_n \}_{n\in \N}$ in $E$ with $v_n\to v$ such that $\bar{x}+\theta_nv_n\in C$ for each $n\in \N$. Thus, $\varphi_i(\bar{x}+\theta_nv_n)\le 0$ for each $i\in I(\bar{x})$. Since $\varphi_i$ is strictly differentiable (and hence Gateaux differentiable) at $\bar{x}$, we have 
$$
\langle \nabla \varphi_i(\bar{x}),v \rangle=\lim_{n\to \infty}\frac{\varphi_i(\bar{x}+\theta_nv)-\varphi_i(\bar{x})}{\theta_n}\le \lim_{n\to \infty}\left( \frac{\varphi_i(\bar{x}+\theta_nv_n)}{\theta_n}+\alpha_i\| v_n -v \| \right) \le 0
$$
for each $i\in I(\bar{x})$, where $\alpha_i$ is a Lipschitz modulus of $\varphi_i$. Hence, $T_C(\bar{x})\subset \{ v\in E\mid  \langle \nabla \varphi_i(\bar{x}),v \rangle\le 0 \ \forall i\in I(\bar{x}) \}$. Since $0\not\in \mathrm{co}\,\{ \nabla \varphi_i(\bar{x})\mid i\in I(\bar{x}) \}$, by the separation theorem, there exists $\bar{v}\in E$ such that $\sup_{i\in I(\bar{x})}\langle \nabla \varphi_i(\bar{x}),\bar{v} \rangle=:-\varepsilon<0$. Let $\{ x_n \}_{n\in \N}$ be a sequence in $C$ with $x_n\to \bar{x}$ and $\{ \theta_n \}_{n\in \N}$ be a sequence of positive real numbers with $\theta_n\downarrow 0$. Then $\varphi_i(x_n)\le 0$ and 
\begin{align*}
\limsup_{n\to \infty}\frac{\varphi_i(x_n+\theta_n \bar{v})}{\theta_n}\le \limsup_{n\to \infty}\frac{\varphi_i(x_n+\theta_n \bar{v})-\varphi_i(x_n)}{\theta_n}=\langle \nabla\varphi_i(\bar{x}),\bar{v} \rangle<0.
\end{align*}
Thus, for every $n$ sufficiently large, we have $\varphi_i(x_n+\theta_n\bar{v})<0$ for each $i\in I(\bar{x})$. Consequently, $x_n+\theta_n \bar{v}\in C$ for every $n$ sufficiently large, and hence, $\bar{v}\in T_C(\bar{x})$. Let $v\in E$ be such that $\langle \nabla\varphi_i(\bar{x}),v\rangle\le 0$ for each $i\in I(\bar{x})$. For $\alpha\in (0,1)$ define $v_\alpha:=\alpha \bar{v}+(1-\alpha)v$. Then $\langle \nabla \varphi_i(\bar{x}),v_\alpha \rangle<0$  for each $i\in I(\bar{x})$. Similar to the case for $\bar{v}$, we have $v_\alpha\in T_C(\bar{x})$. Since $v_\alpha\to v$ as $\alpha\to 0^+$, taking the limit yields $v\in T_C(\bar{x})$. Therefore, $T_C(\bar{x})=\{ v\in E\mid \langle \nabla\varphi_i(\bar{x}),v\rangle\le 0 \ \forall i\in I(\bar{x}) \}$.  

Set $A=\sum_{i\in I(\bar{x})}\R_+\nabla \varphi_i(\bar{x})$. Clearly, $\langle x^*,v \rangle\le 0$ for every $x^*\in A$ and $v\in T_C(\bar{x})$, and hence, $A \subset N_C(\bar{x})$. We claim that $A=N_C(\bar{x})$. Toward this end, we first show that $A$ is weakly$^*\!$ closed. Let $\{ x^*_\nu \}$ be a net in $A$ converging weakly$^*\!$ to $x^*$ with $x^*_\nu=\sum_{i\in I(\bar{x})}\lambda^{\nu}_i\nabla\varphi_i(\bar{x})$ for each $\nu$. Then $\langle x^*_\nu,\bar{v} \rangle \to \langle x^*,\bar{v} \rangle$. Since $\sum_{i\in I(\bar{x})}\lambda^{\nu}_i \langle \nabla\varphi_i(\bar{x}),\bar{v} \rangle\le -\varepsilon \sum_{i\in I(\bar{x})}\lambda^{\nu}_i\le 0$ for each $\nu$, the net $\{ (\lambda_{i_1}^\nu,\dots,\lambda_{i_k}^\nu) \}$ is bounded in $\R_+^k$, where $I(\bar{x})=\{i_1,\dots i_k \}$ with $k \le m$. Thus it has a subnet converging to some $(\lambda_{i_1},\dots,\lambda_{i_k})$ in $\R_+^k$. Hence, $x^*_\nu$ converges weakly$^*\!$ to $x^*=\sum_{i\in I(\bar{x})}\lambda_i\nabla\varphi_i(\bar{x})\in A$. Assume for a moment that there exists $p\in N_C(\bar{x})$ such that $p\not\in A$. By the separation theorem, there exists $y\in E$ such that $0\le \sup_{y^*\in A}\langle y^*,y \rangle<\langle p,y \rangle$. Since $A$ is a cone, we must have $\sup_{y^*\in A}\langle y^*,y \rangle=0$ implying that $\langle \nabla\varphi_i(\bar{x}),y \rangle\le 0$ for each $i\in I(\bar{x})$. Therefore, $y\in T_C(\bar{x})$ and $\langle p,y \rangle\le 0$. The obtained contradiction yields $A=N_C(\bar{x})$. \qed

\subsection{Gelfand Integrals of Multifunctions}
\label{subsec1}
Let $ I$ be a nonempty closed subset of the real line $\R$ with the Lebesgue measure and the Lebesgue $\sigma$-algebra $\L$. Denote by $\mathrm{Borel}(E^*,\mathit{w}^*)$ the Borel $\sigma$-\hspace{0pt}algebra of the dual space $E^*$ generated by the weak$^*\!$ topology. A function $f:I\to E^*$ is said to be \textit{weakly$^*\!$ scalarly measurable} if the scalar function $\langle f(\cdot),x\rangle$ is measurable for every $x\in E$. If $E$ is a separable Banach space, then $E^*$ is a locally convex Suslin space under the weak$^*\!$ topology. In this case, a function $f:I\to E^*$ is weakly$^*\!$ scalarly measurable if and only if it is measurable with respect to $\mathrm{Borel}(E^*,\mathit{w}^*)$; see \citet[Theorem III.36]{cv77}. A weakly$^*\!$ scalarly measurable function $f$ is said to be \textit{weakly$^*\!$ scalarly integrable} if $\langle f(\cdot),x \rangle$ is integrable for every $x\in E$. Further, a weakly$^*\!$ scalarly measurable function $f$ is said to be \textit{Gelfand integrable} (or \textit{weakly$^*\!$ integrable}) over a given set $A\in\L$ if there exists $x^*_A\in E^*$ such that $\langle x^*_A,x \rangle=\int_A\langle f(t),x \rangle dt$ for every $x\in E$. The element $x^*_A$ is called the \textit{Gelfand} (or \textit{weak$^*$}) \textit{integral} of $f$ over $A$ and is denoted by $\int_Afdt$. Note that every weakly$^*\!$ scalarly integrable function is Gelfand integrable over $I$ as shown in \citet[Theorem 11.52]{ab06}. 

A multifunction $\Phi:I\rightsquigarrow E^*$ is said to be \textit{upper measurable} if the set $\{ t\in I\mid \Phi(t)\subset O \}$ belongs to $\L$ for every weakly$^*\!$ open subset $O$ of $E^*$; $\Phi$ is said to be \textit{graph measurable} if the set $\mathrm{gph}\,\Phi:=\{ (t,x^*)\in I\times E^*\mid x^*\in \Phi(t) \}$ belongs to $\L\otimes \mathrm{Borel}(E^*,\mathit{w}^*)$; $\Phi$ is said to be \textit{weakly$^*\!$ scalarly measurable} if the scalar function $s(x,\Phi(\cdot)):I\to \R\cup\{ \pm\infty \}$ is measurable for every $x\in E$, where we set $s(x,\emptyset):=-\infty$. When $\Phi$ has empty values on a null set $N$, we can extend it to $I$ with nonempty values at every point by setting $\Phi(t):=\{ 0 \}$ for $t\in N$, preserving its upper, graph, and weak$^*$ scalar measurability. A function $f:I\to E^*$ is called a \textit{selector} of $\Phi$ if $f(t)\in \Phi(t)$ a.e.\ $t\in I$. If $E$ is separable, then an a.e.\ nonempty-valued multifunction $\Phi:I\rightsquigarrow E^*$ with measurable graph in $\L\otimes \mathrm{Borel}(E^*,\mathit{w}^*)$ admits a $\mathrm{Borel}(E^*,\mathit{w}^*)$-\hspace{0pt}measurable selector; see \citet[Theorem III.22]{cv77}. If $E$ is separable and $\Phi$ has a.e.\ nonempty, weakly$^*\!$ compact, convex values, then $\Phi$ is weakly$^*\!$ scalarly measurable if and only if it is upper measurable (see \citet[Theorem 18.31]{ab06}), and in this case, $\Phi$ admits a $\mathrm{Borel}(E^*,\mathit{w}^*)$-\hspace{0pt}measurable (or equivalently, weakly$^*\!$ scalarly measurable) selector; see \citet[Theorem 18.33]{ab06} or \citet[Corollary 3.1]{ckr11}. 

A multifunction $\Phi:I\rightsquigarrow E^*$ with a.e.\ nonempty values is \textit{integrably bounded} if there exists an integrable function $\gamma:I\to \R$ such that $\sup_{x^*\in \Phi(t)}\| x^* \|\le \gamma(t)$ a.e.\ $t\in I$. If $\Phi$ is integrably bounded with measurable graph, then it admits a Gelfand integrable selector whenever $E$ is separable. Denote by $\mathcal{S}^1_\Phi$ the set of Gelfand integrable selectors of $\Phi$. The Gelfand integral of $\Phi$ is conventionally defined as $\int\Phi dt:=\{ \int fdt \mid f\in \mathcal{S}^1_\Phi \}$. If $\Phi$ is an integrably bounded, weakly$^*$ closed, convex-valued multifunction with measurable graph, then $\int\Phi dt$ is nonempty, weakly$^*\!$ compact, and convex with $s(x,\int\Phi dt)=\int s(x,\Phi(t))dt$ for every $x\in E$ whenever $E$ is separable; see \citet[Proposition 2.3 and Theorem 4.5]{ckr11}.

\subsection{Gelfand Integrals of Clarke and Dini--Hadamard Subdifferential Mappings}
\label{subsec2}
An $\L\otimes \mathrm{Borel}(E,\| \cdot \|)$-measurable function $L:I\times E\to \R\cup \{ +\infty \}$ is called a \textit{normal integrand} if $L(t,\cdot)$ is lower semicontinuous on $E$ for every $t\in I$. For a given measurable function $x(\cdot):I\to E$, let $L:I\times E\to \R$ be a function such that (i) $L(\cdot,x)$ is measurable for every $x\in E$; (ii) there exist $\varepsilon>0$ and an integrable function $k:I\to \R$ such that $|L(t,x)-L(t,y)|\le k(t)\| x-y \|$ for every $x,y\in x(t)+\varepsilon B$ and $t\in I$. The Clarke subdifferential mapping $t\rightsquigarrow \partial^\circ_xL(t,x(t))$ is an integrably bounded multifunction from $I$ to $E^*$ with weakly$^*$ compact, convex values. In view of the fact that $L^\circ_x(t,x(t);v)=s(v,\partial^\circ_xL(t,x(t)))$ for every $t\in I$ and $v\in E$, the Clarke subdifferential mapping $\partial^\circ_xL(\cdot,x(\cdot)):I\rightsquigarrow E^*$ is weakly$^*\!$ scalarly measurable if and only if the Clarke directional derivative function $L^\circ_x(\cdot,x(\cdot);v):I\to \R$ is measurable for every $v\in E$. This holds in particular when $E$ is separable (see \citet[Lemma, p.\,78 and the proof of Theorem 2.7.8]{cl83}), and hence, in this case, $\partial^\circ_xL(\cdot,x(\cdot))$ admits a Gelfand integrable selector. 

We summarize the above result on the Gelfand integrability of the Clarke subdifferential mapping together with the results in Subsection \ref{subsec1} as follows.

\begin{prop}[\citet{ckr11,cl83}]
\label{cl}
Let $E$ be a separable Banach space and $I$ be a nonempty closed subset of $\R$. If, for a given measurable function $x:I\to E$, the function $L:I\times E\to \R$ satisfies the following conditions:
\begin{enumerate}[\rm(i)]
\item $L(\cdot,x)$ is measurable for every $x\in E$;
\item There exist $\varepsilon>0$ and an integrable function $k:I\to \R$ such that $|L(t,x)-L(t,y)|\le k(t)\| x-y \|$ for every $x,y\in x(t)+\varepsilon B$ and $t\in I$;
\end{enumerate}
then the Clarke subdifferential mapping $\partial^\circ_xL(\cdot,x(\cdot)):I\rightsquigarrow E^*$ has a Gelfand integrable selector and the Gelfand integral $\int\partial^\circ_xL(t,x(t))dt$ is weakly$^*\!$ compact and convex with 
\begin{align*}
s\left( v,\int_I\partial^\circ_xL(t,x(t))dt \right)=\int_I s\left( v,\partial^\circ_xL(t,x(t)) \right)dt=\int_I L^\circ_x(t,x(t);v)dt
\end{align*}
for every $v\in E$. 
\end{prop}

A similar result holds for Dini--Hadamard subdifferential mappings, but the proof is rather different from the one for Clarke subdifferential mappings since it involves a geometric aspect using the contingent cone and its polar. 

\begin{thm}
\label{thm1}
Under the hypothesis of Proposition \ref{cl}, if $L$ is a normal integrand with $\partial^-_xL(t,x(t))\ne \emptyset$ a.e.\ $t\in I$, then the Dini--Hadamard subdifferential mapping $\partial^-_xL(\cdot,x(\cdot)):I\rightsquigarrow E^*$ has a  Gelfand integrable selector and the Gelfand integral $\int\partial^-_xL(t,x(t))dt$ is weakly$^*\!$ compact and convex with 
\begin{align*}
s\left( v,\int_I\partial^-_xL(t,x(t))dt \right)=\int_I s\left( v,\partial^-_xL(t,x(t)) \right)dt
\end{align*}
for every $v\in E$. 
\end{thm}

Recall that a multifunction $\Gamma:I \rightsquigarrow E$ is said to be \textit{measurable} if the set $\{ t\in I\mid \Gamma(t)\cap O\ne \emptyset \}$ belongs to $\L$ for every open subset $O$ of $E$. Denote by $\Gamma^0:I\rightsquigarrow E^*$ the polar mapping of $\Gamma$ defined by $\Gamma^0(t):=\Gamma(t)^0$. 

\begin{lem}
\label{lem1}
Let $E$ be a separable Banach space. If $\Gamma:I\rightsquigarrow E$ is a measurable multifunction with nonempty closed values, then its polar mapping $\Gamma^0:I\rightsquigarrow E^*$ has the graph in $\L\otimes \mathrm{Borel}(E^*,\mathit{w}^*)$ and $\Gamma^0$ admits a weakly$^*$ scalarly measurable selector.  
\end{lem}

\begin{proof}
Let $\{ g_n \}_{n\in \N}$ be a \textit{Castaing representation} of $\Gamma$, that is, each $g_n:I\to E$ is a measurable selector of $\Gamma$ such that $\mathrm{cl}\{ g_n(t)\mid n\in \N \}=\Gamma(t)$ for every $t\in I$. Since $(t,x^*)\mapsto \langle x^*,g_n(t)\rangle$ is $\L\otimes\mathrm{Borel}(E^*,\mathit{w}^*)$-measurable for each $n\in \N$ (see \citet[Theorem III.36]{cv77}) and $s(x^*,\Gamma(t))=\sup_n\langle x^*,g_n(t)\rangle$ for every $x^*\in E^*$ and $t\in I$, $\mathrm{gph}\,\Gamma^0$ is $\L\otimes \mathrm{Borel}(E^*,\mathit{w}^*)$-measurable. Therefore, $\Gamma^0$ admits a $\mathrm{Borel}(E^*,\mathit{w}^*)$-measurable, and hence, weakly$^*$ scalarly measurable selector.
\end{proof}

\begin{proof}[Proof of Theorem \ref{thm1}]
Define the multifunction $\Gamma:I\rightsquigarrow E\times \R$ by 
$$
\Gamma(t):=K_{\mathrm{epi}\,L(t,\cdot)}(x(t),L(t,x(t))). 
$$
Since $L$ is a normal integrand, the epigraph mapping $t\rightsquigarrow \mathrm{epi}\,L(t,\cdot)$ is a nonempty, closed-valued multifunction with its graph in $\L\otimes\mathrm{Borel}(E\times \R,\| \cdot \|)$; see \citet[Lemma VII.1]{cv77}. It follows from \citet[Theorem 8.5.1]{af90} that $\Gamma$ is a measurable multifunction with nonempty closed values. Then by Lemma \ref{lem1}, the polar mapping $\Gamma^0:I\rightsquigarrow E^*\times \R$ of $\Gamma$ has the graph in $\L\otimes \mathrm{Borel}(E^*\times \R,\mathit{w}^*)$. Define the multifunction $\Phi:I\rightsquigarrow E^*\times \R$ by 
$$
\Phi(t):=\Gamma^0(t)\cap (E^*\times \{ -1 \})=\left\{ (x^*,-1)\in K_{\mathrm{epi}\,L(t,\cdot)}(x(t),L(t,x(t)))^0 \right\}.
$$
Then $\Phi(t)\ne \emptyset$ a.e.\ $t\in I$ and $\mathrm{gph}\,\Phi$ belongs to $\L\otimes \mathrm{Borel}(E^*\times \R,\mathit{w}^*)$. Therefore, $\Phi$ admits a weakly$^*\!$ scalarly measurable selector, and hence, there exists a weakly$^*\!$ scalarly measurable function $f:I\to E^*$ such that $(f(t),-1)\in \Phi(t)$ a.e.\ $t\in I$. Since $f(t)\in \partial^-_xL(t,x(t))$ and the Dini--Hadamard subdifferential mapping $\partial^-_xL(\cdot,x(\cdot))$ is integrably bounded, $f$ is Gelfand integrable. Under the assumptions of the theorem, $\partial^-_xL(t,x(t))$ is nonempty, weakly$^*\!$ compact, and convex a.e.\ $t\in I$. Therefore, the Gelfand integral of $\partial^-_xL(\cdot,x(\cdot))$ is nonempty, weakly$^*\!$ compact and convex, and the desired equality holds as noted in Subsection \ref{subsec1}. 
\end{proof}

\begin{rem}
Note that unlike Clarke directional derivatives, the lack of convexity of the function $v\mapsto L^-_x(t,x(t);v)$ leads to the failure of the equality $L^-_x(t,x(t);v)=s(v,\partial^-_xL(t,x(t)))$ even if $\partial^-_xL(t,x(t))$ is nonempty. This is a  disadvantage of the use of Dini--Hadamard subdifferentials because $\partial^-_xL(t,x(t))$ may be empty even if $L(t,\cdot)$ is Lipschitz on $x(t)+\varepsilon B$. On the other hand, if $L(t,\cdot)$ is strictly differentiable at $x(t)$, then its Dini--Hadamard subdifferential is the singleton $\{ \nabla_xL(t,x(t)) \}$. See also Remark \ref{rem2} for a further discussion. 
\end{rem}

\section{Appendix II}
\subsection{Lipschitz Continuity of the Value Function}
\label{apdx3}
The following result is a special case of \citet[Theorem 1.2]{fr90}, which is an infinite-dimensional analogue of the celebrated Filippov theorem; see \citet{fi67}. Note that the solution concept adopted in \citet{fr90} is a mild solution to an evolution differential inclusion involving semigroups of unbounded linear operators. As in our case the semigroup is given by the identity operator, it follows from the Lebesgue differentiation theorem that the mild solution has a strong derivative that is Bochner integrable whenever $E$ is separable.    

\begin{lem}[\citet{fr90}]
\label{fr1}
Let $E$ be a separable Banach space and $[t_0,t_1]$ be any closed interval in $\R_+$. If $\mathrm{(H_4)}$, $\mathrm{(H_5)}$, and $\mathrm{(H_6)}$ hold, and $y(\cdot)\in W^{1,1}([t_0,t_1],E)$ is such that $t\mapsto d_{\Gamma(t,y(t))}(\dot{y}(t))$ is integrable with $y(t_0)=\xi\in E$, then for every $\xi'\in E$ and $\varepsilon>0$ there exists $x(\cdot)\in W^{1,1}([t_0,t_1],E)$ such that:
\begin{enumerate}[\rm (i)]
\item $\dot{x}(t)\in \Gamma(t,x(t))$ a.e.\ $t\in [t_0,t_1]$ with $x(t_0)=\xi'$;
\item $\| x(t)-y(t) \|\le \displaystyle \exp\left( \int_{t_0}^t\gamma(s)ds \right) \\ \hspace{3cm} \times \left( \|\xi-\xi'\|+\int_{t_0}^td_{\Gamma(s,y(s))}(\dot{y}(s))ds+\varepsilon(t-t_0) \right) 
$ \\
for every $t\in [t_0,t_1]$;
\item $\| \dot{x}(t)-\dot{y}(t) \|\le \displaystyle\exp\left( \int_{t_0}^t\gamma(s)ds \right)\gamma(t)(\|\xi-\xi'\|+\varepsilon(t-t_0))+d_{\Gamma(t,y(t))}(\dot{y}(t))+\varepsilon$ \\ a.e.\ $t\in [t_0,t_1]$. 
\end{enumerate}
\end{lem}

\begin{proof}[Proof of Theorem \ref{thm2}]
Take any $(t,\xi)\in \R_+\times E$. Since, by $\mathrm{(H_6)}$, every admissible trajectory $x(\cdot)\in \A_{(t,\xi)}$ satisfies the inequality $\| \dot{x}(s) \|\le \gamma(s)+\gamma(s)\| x(s)\|$ a.e.\ $s\in [t,\infty)$ by $\mathrm{(H_6)}$, the Gronwall's inequality yields 
$$
\| x(s)\|\le \exp\left( \int_0^s \gamma(\tau)d\tau \right)\left( \| \xi \|+\int_0^s\gamma(\tau)d\tau \right)=:\gamma_{\| \xi \|}(s)<\infty
$$
for every $s\in [t,\infty)$. It follows from $\mathrm{(H_3)}$ and $\mathrm{(H_7)}$ that 
$$
|L(s,x(s),\dot{x}(s))|\le  l_1(s)+l_1(s)\gamma_{\| \xi \|}(s)+l_2(s)(\gamma(s)+\gamma(s)\gamma_{\| \xi \|}(s))=:k_{\| \xi \|}(s)
$$
and $k_{\| \xi \|}(\cdot)$ is integrable over $[t,\infty)$. Therefore, $V$ is bounded. Furthermore, it follows from $|\int_t^\infty L(s,x(s),\dot{x}(s))ds|\le \int_t^\infty k_{\| \xi \|}(s)ds$ that for every $\varepsilon>0$ there exists $t_0\in \R_+$ such that $|\int_T^\infty L(s,x(s),\dot{x}(s))ds|<\varepsilon$ for every $T>t_0$ and $x(\cdot)\in \A_{(t,\xi)}$. This implies that $\sup_{x(\cdot)\in \A_{(t,\xi)}}|V(T,x(T))|\to 0$ as $T\to \infty$ for every $(t,\xi)\in \R_+\times E$.

Next, we demonstrate the Lipschitz continuity of $V(t,\cdot)$. Let $\xi,\xi'\in E$ be arbitrary. Take any $\varepsilon>0$ and $T\in [t,\infty)$. Then by $\mathrm{(H_1)}$ and the Bellman principle of optimality, there exists $x(\cdot)\in \A_{(t,\xi)}$ such that $\int_t^TL(s,x(s),\dot{x}(s))ds+V(T,x(T))<V(t,\xi)+\varepsilon$. It follows from Lemma \ref{fr1} that there exists $x^T(\cdot)\in W^{1,1}([t,T],E)$ such that:
\begin{enumerate}[\rm (i)]
\item $\dot{x}^T(s)\in \Gamma(s,x^T(s))$ a.e.\ $s\in [t,T]$ with $x^T(t)=\xi'$; 
\item $\| x^T(s)-x(s) \|\le \displaystyle\exp\left( \int_t^s\gamma(\tau)d\tau \right)(\|\xi'-\xi\|+\varepsilon(s-t))$ for every $s\in [t,T]$; 
\item $\| \dot{x}^T(s)-\dot{x}(s) \|\le \displaystyle\exp\left( \int_t^s\gamma(\tau)d\tau \right)\gamma(s)(\|\xi'-\xi\|+\varepsilon(s-t))+\varepsilon$ a.e.\ $s\in [t,T]$. 
\end{enumerate}
Take any $x_T(\cdot)\in \A_{(T,x^T(T))}$ and define $y_T(\cdot)\in \A_{(t,\xi')}$ by $y_T(\cdot)=x^T(\cdot)$ on $[t,T]$ and $y_T(\cdot)=x_T(\cdot)$ on $(T,\infty)$. As observed in the above, we obtain 
$$
\lim_{T\to \infty}|V(T,x^T(T))|=\lim_{T\to \infty}|V(T,y_T(T))|\le \lim_{T\to \infty}\sup_{z(\cdot)\in \A_{(t,\xi')}}|V(T,z(T))|=0.
$$ 
Similarly, $\lim_{T\to \infty}V(T,x(T))=0$. By the Bellman principle of optimality, we have
\begin{align*}
V(t,\xi')-V(t,\xi)
& \le \int_t^T L(s,x^T(s),\dot{x}^T(s))ds+V(T,x^T(T)) \\
& \qquad -\int_t^T L(s,x(s),\dot{x}(s))ds-V(T,x(T))+\varepsilon  \\ 
& \le \int_t^T\left[ l_1(s)\| x^T(s)-x(s) \|+l_2(s)\| \dot{x}^T(s)-\dot{x}(s) \| \right]ds \\
& \qquad +V(T,x^T(T))-V(T,x(T))+\varepsilon \\ 
& \le k_1(t)(\| \xi'-\xi \|+\varepsilon(T-t))+k_2(t)(\| \xi'-\xi \|+\varepsilon(T-t)) \\
& \qquad +V(T,x^T(T))-V(T,x(T))+\varepsilon, 
\end{align*}
where we set in the last inequality $k_1(t):=\int_t^\infty\exp(\int_t^s\gamma(\tau)d\tau) l_1(s)ds$ and $k_2(t):=\int_t^\infty\exp(\int_t^s\gamma(\tau)d\tau)l_2(s)\gamma(s)ds$. Since $\varepsilon$ is arbitrary, we obtain 
$$
V(t,\xi')-V(t,\xi)\le k(t)\| \xi'-\xi \|+V(T,x^T(T))-V(T,x(T))
$$
for every $T\in [t,\infty)$ with $k(t):=k_1(t)+k_2(t)$. Then $k:\R_+\to \R_+$ is a continuous decreasing function with $k(t)\to 0$ as $t\to \infty$. Letting $T\to \infty$ in this inequality yields $V(t,\xi')-V(t,\xi)\le k(t)\| \xi'-\xi \|$. Since the role of $\xi$ and $\xi'$ is interchangeable in the above argument, we have demonstrated that $V(t,\cdot)$ is Lipschitz of rank $k(t)$ on $E$ for every $t\in \R_+$.

Finally, we show the lower semicontinuity of $V$. Toward this end, fix $t\in \R_+$. It suffices to show that $V(\cdot,\xi)$ is lower semicontinuous on $\R_+$ for every $\xi\in E$. Indeed, since we have
$$
V(t',\xi)-k(t')\| \xi'-\xi \|\le V(t',\xi')
$$
for every $(t',\xi')\in \R_+\times E$, taking the limit inferior in the both sides of the above inequality yields 
$$
\liminf_{t'\to t}V(t',\xi)\le \liminf_{(t',\xi')\to (t,\xi)}V(t',\xi'). 
$$

Take any $t'\in \R_+$. If $t'\in [t,\infty)$, then 
\begin{equation}
\label{eq6}
V(t,\xi)\le \int_t^{t'}L(s,x(s),\dot{x}(s))ds+V(t',x(t'))
\end{equation} 
for every $x(\cdot)\in \A_{(t,\xi)}$. Since $V(t',\cdot)$ is Lipschitz of rank $k(t')\le k(t)$, we obtain $|V(t',x(t'))-V(t',\xi)|\le k(t)\| x(t')-\xi \|\to 0$ as $t'\downarrow t$. Taking the limit inferior in the both sides of \eqref{eq6} yields 
$$
V(t,\xi)\le \liminf_{t'\downarrow t}V(t',\xi). 
$$
Similarly, if $t'\in [0,t)$ with $t>0$, then for every $\varepsilon>0$ there exists $y(\cdot)\in \A(t',\xi)$ such that 
\begin{equation}
\label{eq7}
\int_{t'}^tL(s,y(s),\dot{y}(s))ds+V(t,y(t))\le V(t',\xi)+\varepsilon
\end{equation}
Since $V(t,\cdot)$ is Lipschitz of rank $k(t)$, we have 
\begin{align*}
& |V(t,y(t))-V(t,\xi)| \\ {}\le{} 
& k(t)\| y(t)-\xi \|=k(t)\left\| \int_{t'}^t\dot{y}(s)ds \right\| \\ {}\le{} 
& k(t)\int_{t'}^t\left[ \gamma(s)+\gamma(s)\| y(s) \| \right]ds\le k(t)\int_{t'}^t\left[ \gamma(s)+\gamma(s)\gamma_{\| \xi \|}(s) \right]ds\to 0
\end{align*}
as $t'\uparrow t$ and
$$
\left| \int_{t'}^tL(s,y(s),\dot{y}(s))ds \right|\le \int_{t'}^tk_{\| \xi \|}(s)ds\to 0 
$$
as $t'\uparrow t$. Hence, taking the limit inferior in \eqref{eq7} yields 
$$
V(t,\xi)\le \liminf_{t'\uparrow t}V(t',\xi)+\varepsilon.
$$
Since $\varepsilon$ is arbitrary, we obtain 
$$
V(t,\xi)\le \liminf_{t'\to t}V(t',\xi). 
$$
Therefore, $V(\cdot,\xi)$ is lower semicontinuous at every $t\in \R_+$.
\end{proof}

\subsection{Subdifferentiability of the Value Function}
\label{apdx4}
Denote by $o(h)>0$ the Landau symbol with $\lim_{h\downarrow 0}h^{-1}o(h)=0$.

\begin{lem}
\label{lem2}
Suppose that $\mathrm{(H_4)}$, $\mathrm{(H_5')}$, and $\mathrm{(H_6)}$ hold. Let $t\in \R_+$ be such that the strong derivative $\dot{x}_0(t)$ exists and $v\in \Gamma(t,x_0(t))$ be arbitrarily fixed. Then for every $h>0$ there exists $x_h(\cdot)\in W^{1,1}([t,t+h],E)$ such that:
\begin{enumerate}[\rm(i)]
\item $\dot{x}_h(s)\in \Gamma(s,x_h(s))$ a.e.\ $s\in [t,t+h]$ with $x_h(t+h)=x_0(t+h)$; 
\item $\|x_h(t)-x_0(t)-h(\dot{x}_0(t)-v) \|=o(h)$; 
\item $\| \dot{x}_h(\cdot)-v\|_{L^1([t,t+h])}=o(h)$.
\end{enumerate} 
\end{lem}

\begin{proof}
Define $y_h(s):=x_0(t+h)-sv$ for $s\in [0,h]$ and the multifunction $\Gamma_h:\R\times E\rightsquigarrow E$ by
$$
\Gamma_h(s,x):=
\begin{cases}
-\Gamma(t+h-s,x) & \text{if $s\in [0,h]$}, \\
-\Gamma(t,x) & \text{if $s\in (h,\infty)$}.
\end{cases}
$$ 
By $\mathrm{(H_6)}$, we have $\Gamma_h(s,y_h(s))\subset \Gamma_h(s,x_0(t))+\gamma(t+h-s)(\| x_0(t+h)-x_0(t) \|+h\| v \|)B$ for every $s\in [0,h]$, which yields the inequality
$$
d_{\Gamma_h(s,y_h(s))}(\dot{y}_h(s))\le d_{\Gamma_h(s,x_0(t))}(-v)+\gamma(t+h-s)(\| x_0(t+h)-x_0(t) \|+h\| v \|).
$$
Since the multifunction $(s,h)\rightsquigarrow \Gamma_h(s,x_0(t))$ is lower semicontinuous, the distance function $(s,h)\mapsto d_{\Gamma_h(s,x_0(t))}(-v)$ is upper semicontinuous; see \citet[Corollary 1.4.17]{af90}. Let $\varphi(s,h):=d_{\Gamma_h(s,x_0(t))}(-v)$ and $\hat{\varphi}(h):=\sup_{s\in [0,h]}\varphi(s,h)$. Then for every $\varepsilon>0$ and $h\ge 0$ there exists $s_h\in [0,h]$ such that $\hat{\varphi}(h)<\varphi(s_h,h)+\varepsilon$. Since $s_h\to 0$ as $h\downarrow 0$ and $\varphi$ is upper semicontinuous at the origin with $\varphi(0,0)=d_{\Gamma(t,x_0(t))}(v)=0$, taking the limit superior of the above inequality yields $\limsup_{h\downarrow 0}\hat{\varphi}(h)\le \limsup_{h\downarrow 0}\varphi(s_h,h)+\varepsilon\le \varepsilon$. Since $\varepsilon$ is arbitrary, we have $\lim_{h\downarrow 0}\hat{\varphi}(h)=0$. Consequently, it follows from the inequality $\| x_0(t+h)-x_0(t) \|\le h\|\dot{x}_0(t)\|+o(h)$ that 
$$
d_{\Gamma_h(s,y_h(s))}(\dot{y}_h(s))\le \hat{\varphi}(h)+\gamma(t+h-s)(h(\|\dot{x}_0(t)\|+\| v \|)+o(h)).
$$

By Lemma \ref{fr1} applied with $\varepsilon=h$, there exists $z(\cdot)\in W^{1,1}([0,h],E)$ such that $\dot{z}(s)\in \Gamma_h(s,z(s))$ a.e.\ $s\in [0,h]$ with $z(0)=y_h(0)=x_0(t+h)$ satisfying 
\begin{align*}
\| z(s)-y_h(s) \|
&\le \exp\left( \int_0^h\gamma(t+h-\tau)d\tau \right)\left(\int_0^hd_{\Gamma_h(\tau,y_h(\tau))}(\dot{y}_h(\tau))d\tau+h^2 \right) \\ 
& \le\exp\left( \int_t^{t+h}\gamma(\tau)d\tau \right) \\ 
& \quad \times \left( h\hat{\varphi}(h)+(h(\| \dot{x}_0(t) \|+\| v \|)+o(h)) \int_t^{t+h}\gamma(\tau)d\tau+h^2\right)=o(h)
\end{align*}
and 
\begin{align*}
\| \dot{z}(s)-\dot{y}_h(s) \|
& \le h^2\exp\left( \int_0^h\gamma(t+h-\tau)d\tau \right)\gamma(t+h-s)+d_{\Gamma_h(s,y_h(s))}(\dot{y}_h(s))+h \\
& \le h^2\exp\left( \int_t^{t+h}\gamma(\tau)d\tau \right)\gamma(t+h-s)+\hat{\varphi}(h) \\ 
& \qquad +\gamma(t+h-s)(h(\|\dot{x}_0(t)\|+\| v \|)+o(h))+h 
\end{align*}
for a.e.\ $s\in [0,h]$. Integrating the both sides of the above inequality over $[0,h]$ yields 
\begin{align*}
& \| \dot{z}(\cdot)-\dot{y}_h(\cdot)\|_{L^1([0,h])} \\{}={}
& \int_t^{t+h}\gamma(\tau)d\tau\left( h^2\exp\left( \int_t^{t+h}\gamma(\tau)d\tau \right)+h(\|\dot{x}_0(t)\|+\| v \|)+o(h) \right) \\ {}
& \qquad +h\hat{\varphi}(h)+h^2= o(h). 
\end{align*}
Set $x_h(\tau):=z(t+h-\tau)$ for $\tau\in [t,t+h]$. Then $x_h(t+h)=z(0)=x_0(t+h)$ and $\dot{x}_h(\tau)=-\dot{z}(t+h-\tau)\in -\Gamma_h(t+h-\tau,x_h(\tau))=\Gamma(\tau,x_h(\tau))$ a.e.\ $\tau\in [t,t+h]$. Thus, condition (i) is verified. Since $\| x_h(t)-y_h(h) \|=\| z(h)-y_h(h) \|=o(h)$, we have $\|x_h(t)-x_0(t+h)+hv\|=o(h)$. Hence, $\|x_h(t)-x_0(t)-h(\dot{x}_0(t)-v) \|\le \|x_h(t)-x_0(t+h)+hv\|+\| x_0(t+h)-x_0(t)-h\dot{x}_0(t)\|=o(h)$ and we obtain condition (ii). In view of $\dot{x}_h(\tau)=-\dot{z}(t+h-\tau)$ and $\dot{y}_h(t+h-\tau)=-v$, we obtain $\| \dot{x}_h(\cdot)-v\|_{L^1([t,t+h])}=\| \dot{y}_h(\cdot)-\dot{z}(\cdot)\|_{L^1([0,h])}=o(h)$, which implies condition (iii).
\end{proof}

\begin{proof}[Proof of Theorem \ref{thm3}]
(i): Let $t>0$, $v\in \Gamma(t,x_0(t))$, and $x_h(\cdot)\in W^{1,1}([t,t+h],E)$ be as in the claim of Lemma \ref{lem2}. By condition (ii) of Lemma \ref{lem2}, for every $s\in [t,t+h]$ we have 
\begin{align*}
\| x_h(s)-x_0(t) \|
& \le \| x_h(s)-x_h(t) \|+\| x_h(t)-x_0(t) \| \\
& \le h\| \dot{x}_h(t) \|+o(h)+h\| \dot{x}_0(t)-v \|  \\
& \le h(\gamma(t)+\gamma(t)\| x_h(t) \|)+h\| \dot{x}_0(t)-v \|+o(h) \\
& \le h(\gamma(t)+\gamma(t)(\| x_0(t) \|+h\| \dot{x}_0(t)-v \|+o(h))+h\| \dot{x}_0(t)-v \| \\
& \qquad +o(h) \\
& =h(\gamma(t)+\gamma(t)\| x_0(t) \|)+h\| \dot{x}_0(t)-v \|+o(h),
\end{align*}
which yields the following estimates: 
\begin{align*}
& \left| \int_t^{t+h}L(s,x_h(s),\dot{x}_h(s))ds-\int_t^{t+h}L(s,x_0(t),v)ds \right| \\ {}\le{} 
& \int_t^{t+h}\left[ l_1(s)\| x_h(s)-x_0(t) \|+l_2(s) \| \dot{x}_h(s)-v \| \right]ds \\ {}\le{}
& (h(\gamma(t)+\gamma(t)\| x_0(t) \|)+h\| \dot{x}_0(t)-v \|+o(h))\int_t^{t+h}l_1(s)ds \\
& \quad +\sup_{s\in [t,t+h]}l_2(s) \| \dot{x}_h(\cdot)-v\|_{L^1([t,t+h])}=o(h).
\end{align*}
By the separability of $E$ and \citet[Theorem 2.5]{fpr95}, there exists a subset $I$ of $\R_+$ such that the Lebesgue measure of its complement $\R_+\setminus I$ is zero with $\lim_{h \downarrow 0}h^{-1}\int _t^{t+h} L(s,x,v)ds=L(t,x,v)$ for every $(t,x,v)\in I\times E\times E$. We thus obtain
$$
\lim_{h\downarrow 0}\frac{1}{h}\int_t^{t+h}L(s,x_h(s),\dot{x}_h(s))ds=\lim_{h\downarrow 0}\frac{1}{h}\int_t^{t+h}L(s,x_0(t),v)ds=L(t,x_0(t),v)
$$
for every $t\in I$. Let $t\in \R_+$ be a Lebesgue point of $L(\cdot,x_0(\cdot),\dot{x}_0(\cdot))$. By the Bellman principle of optimality, we have
$$
V(t,x_h(t))\le \int_t^{t+h}L(s,x_h(s),\dot{x}_h(s))ds+V(t+h,x_0(t+h)). 
$$
Subtracting $V(t,x_0(t))=\int_t^{t+h}L(s,x_0(s),\dot{x}_0(s))ds+V(t+h,x_0(t+h))$ from the both sides of the above inequality yields
\begin{align*}
& \int_t^{t+h}L(s,x_h(s),\dot{x}_h(s))ds-\int_t^{t+h}L(s,x_0(s),\dot{x}_0(s))ds \\ {}\ge{} 
& V(t,x_h(t))-V(t,x_0(t))  {}\ge{} V(t,x_0(t)+h(\dot{x}_0(t)-v))-V(t,x_0(t))-o(h)
\end{align*}
because of the Lipschitz continuity of $V(t,\cdot)$. Dividing the both sides of the above inequality by $h>0$ and taking the limit inferior as $h\to 0$ yield the inequality 
$$
V^-_x(t,x_0(t);\dot{x}_0(t)-v)\le L(t,x_0(t),v)-L(t,x_0(t),\dot{x}_0(t)). 
$$
Since $v\in \Gamma(t,x_0(t))$ is arbitrary, the above holds true for any such $v$. 

(ii): Take any $x^*\in \partial^-_xV(t,x_0(t))$. If $u\in K_{\Gamma(t,x_0(t))}(\dot{x}_0(t))$, then there exist a sequence $\{\theta_n\}_{n\in \N}$ of positive real numbers with $\theta_n\to 0$ and a sequence $\{ u_n \}_{n\in \N}$ in $E$ with $u_n\to u$ such that $\dot{x}_0(t)+\theta_nu_n\in \Gamma(t,x_0(t))$ for each $n\in \N$. Since it follows from condition (i) that 
$V^-_x(t,x_0(t);-\theta_nu_n)\le L(t,x_0(t),\dot{x}_0(t)+\theta_nu_n)-L(t,x_0(t),\dot{x}_0(t))$, we have 
$$
\langle x^*,-u_n \rangle\le \frac{L(t,x_0(t),\dot{x}_0(t)+\theta_nu_n)-L(t,x_0(t),\dot{x}_0(t))}{\theta_n}.
$$
Letting $n\to \infty$ in the both sides of the above inequality yields 
\begin{equation}
\label{eq11}
\langle -x^*,u \rangle\le L^+_y(t,x_0(t),\dot{x}_0(t);u)\le L^\circ_y(t,x_0(t),\dot{x}_0(t);u)
\end{equation}
for every $u\in K_{\Gamma(t,x_0(t))}(\dot{x}_0(t))$. Suppose, by way of contradiction, that $-x^*\not \in \partial^\circ_yL(t,x_0(t),\dot{x}_0(t))+N_{\Gamma(t,x_0(t))}(\dot{x}_0(t))$. Since $\partial^\circ_yL(t,x_0(t),\dot{x}_0(t))$ is weakly$^*\!$ compact and convex and $N_{\Gamma(t,x_0(t))}(\dot{x}_0(t))$ is weakly$^*\!$ closed and convex, $\partial^\circ_yL(t,x_0(t),\dot{x}_0(t))+N_{\Gamma(t,x_0(t))}(\dot{x}_0(t))$ is weakly$^*\!$ closed and convex. Then by the separation theorem, there exists $v\in E$ such that 
$$
\langle -x^*,v \rangle>\sup_{y^*\in \partial^\circ_yL(t,x_0(t),\dot{x}_0(t))}\langle y^*,v\rangle+\sup_{z^*\in N_{\Gamma(t,x_0(t))}(\dot{x}_0(t))}\langle z^*,v \rangle.
$$ 
Since $N_{\Gamma(t,x_0(t))}(\dot{x}_0(t))$ is a cone in $E^*$, we must have $\langle z^*,v\rangle\le 0$ for every $z^*\in N_{\Gamma(t,x_0(t))}(\dot{x}_0(t))$. This means that $v\in T_{\Gamma(t,x_0(t))}(\dot{x}_0(t))\subset K_{\Gamma(t,x_0(t))}(\dot{x}_0(t))$ by the bipolar theorem; see \citet[Theorem 2.4.3]{af90}. Since the support function of the Clarke subdifferential $\partial^\circ_yL(t,x_0(t),\dot{x}_0(t))$ coincides with the Clarke directional derivative $L^\circ_y(t,x_0(t),\dot{x}_0(t);v)$, the inequality above finally implies that $\langle -x^*,v \rangle>L^\circ_y(t,x_0(t),\dot{x}_0(t);v)$ with $v\in K_{\Gamma(t,x_0(t))}(\dot{x}_0(t))$, in contradiction with inequality \eqref{eq11}. Consequently, we have $-x^*\in \partial^\circ_yL(t,x_0(t),\dot{x}_0(t))+N_{\Gamma(t,x_0(t))}(\dot{x}_0(t))$. If $L(t,x_0(t),\cdot)$ is Gateaux differentiable at $\dot{x}_0(t)$, then \eqref{eq11} can be replaced by the inequality 
$$
\langle -x^*,u \rangle\le \nabla_y L(t,x_0(t),\dot{x}_0(t);u)
$$ 
for every $u\in K_{\Gamma(t,x_0(t))}(\dot{x}_0(t))$. This implies that the above argument is also valid when we replace $\partial^\circ_yL(t,x_0(t),\dot{x}_0(t))$ and $L^\circ_y(t,x_0(t),\dot{x}_0(t);v)$ respectively by $\nabla_y L(t,x_0(t),\dot{x}_0(t))$ and $\langle \nabla_y L(t,x_0(t),\dot{x}_0(t)),v \rangle$.
\end{proof}

\subsection{Proof of Necessary Conditions for Optimality}
\label{apdx2}
It should be underlined that unlike the real-valued case, locally absolutely continuous functions with values in Banach spaces fail to be strongly differentiable almost everywhere; see \citet[Examples 1 and 2]{pu88} or \citet[Example 4.2]{de92} for such examples. The failure of the strong differentiability of locally absolutely continuous functions disappears under the reflexivity assumption. Specifically, every locally absolutely function $p:\R_+\to E^*$ has the Bochner integrable strong derivative $\dot{p}(t)$ a.e.\ $t\in \R_+\setminus \{ 0 \}$ with $p(t)=\int_0^t\dot{p}(s)ds+p(0)$ for every $t\in \R_+$ whenever $E$ is reflexive; see \citet[Lemma, p.\,505]{ko67}. 

We construct an adjoint variable $p:\R_+\to E^*$ as a locally absolutely continuous function to express optimality conditions. However, we dispense with the reflexivity of $E$. The weak$^*\!$ differentiability of locally absolutely continuous functions is fundamental in the sequel and is virtually contained in the argument of the proof of \citet[Lemma]{ko67}. We provide a proof for the sake of completeness to make clear why the reflexivity of $E$ is irrelevant to weak$^*\!$ differentiability. See also \citet[Theorem 3.5]{ak00} for a strengthened version of the weak$^*\!$ differentiability of Lipschitz functions.  

\begin{lem}[\citet{ko67}]
\label{ko}
Let $E$ be a separable Banach space. Then every locally absolutely continuous function $p:\R_+\to E^*$ possesses the  weak$^*\!$ derivative $\dot{p}(t)$ a.e.\ $t\in \R_+$.  
\end{lem}

\begin{proof}
Define the variation of $p:\R_+\to E^*$ over the compact interval $[0,\tau]$ by $\mathrm{var}(p,[0,\tau]):=\sup\sum_{i=1}^n\| p(t_i)-p(\tau_i)\|$, where the supremum is taken over all finite sets of points $t_i,\tau_i\in [0,\tau]$ with $0\le t_1<\tau_1\le t_2<\tau_2\le \cdots\le t_n<\tau_n\le \tau$. Since $p$ is locally absolutely continuous, $\mathrm{var}(p,[0,\tau])<\infty$ for every $\tau>0$. Define $p_h(t):=h^{-1}(p(t+h)-p(t))$ for $t\in [0,\tau]$ and $h\ne 0$, and $P^+(t):=\limsup_{h\downarrow 0}\| p_h(t) \|$ and $P^-(t):=\limsup_{h\uparrow 0}\| p_h(t) \|$. Since $p$ is continuous, so is $p_h$. Hence, $P^+$ and $P^-$ are measurable on $[0,\tau]$. We claim that $P^+(t)$ and $P^-(t)$ are finite a.e.\ $t\in [0,\tau]$. Suppose to the contrary that the Lebesgue measure $\lambda$ of the set $\{t\in [0,\tau]\mid P^+(t)=\infty \}$ is positive. Let 
$$
A_n:=\left\{ t\in [0,\tau]\left| \begin{array}{l} \hspace{0.2cm} \sup \left[ \left\| \dfrac{p(t+h)-p(t)}{h} \right\| \left|\ h{}\ge{}\dfrac{1}{n},\,0\le t<t+h\le \tau \right.\right]  \\ \ge \dfrac{2}{\lambda} \mathrm{var}(p,[0,\tau]) \end{array} \right.\hspace{-0.2cm}\right\}.
$$
Then each $A_n$ is a closed set and $\{t\in [0,\tau]\mid P^+(t)=\infty \}\subset \bigcup_{n\in \N} A_n$. Since $\{ A_n \}_{n\in \N}$ is an increasing sequence, $|A_n|>\lambda/2$ for some $n\in \N$, where $|A_n|$ denotes the Lebesgue measure of $A_n$. Let $\{ t_i \}_{i\in \N}$ and $\{ h_i \}_{i\in \N}$ be defined inductively by 
\begin{align*}
& t_1:=\inf A_n, \\
& t_{i+1}:=\inf\{ t\in A_n\mid t\ge t_i+h_i \}, \text{ and} \\
& h_i:=\sup\left\{ h>0\mid t_i+h\le \tau,\,\left\| \frac{p(t_i+h)-p(t_i)}{h} \right\|\ge \frac{2}{\lambda}\mathrm{var}(p,[0,\tau]) \right\}.
\end{align*}
Then by construction, we have $A_n\subset\bigcup_{i\in \N}[t_i,t_i+h_i]$, and hence, $|A_n|\le \sum_{i\in \N}h_i$. Consequently, 
$$
\sum_{i\in \N}\| p(t_i+h_i)-p(t_i) \|\ge \frac{2}{\lambda}\mathrm{var}(p,[0,\tau])\sum_{i\in \N}h_i\ge \frac{2}{\lambda}\mathrm{var}(p,[0,\tau])|A_n|>\mathrm{var}(p,[0,\tau]),
$$
a contradiction. In the same way we show that the set $\{t\in [0,\tau]\mid P^-(t)=\infty \}$ is of Lebesgue measure zero. Therefore, there exists a null set $N_0\subset [0,\tau]$ such that for every $t\in [0,\tau]\setminus N_0$ the set $\{ p_h(t)\mid h\ne 0 \}$ is bounded in $E^*$. 

In view of the separability of $E$, there is a countable dense subset $\{ v_i \}_{i\in \N}$ of $E$. Since each scalar function $\varphi_i(t):=\langle p(t),v_i \rangle$ is absolutely continuous on $[0,\tau]$, its derivative $\dot{\varphi}_i(t)$ exists except at a point of a null set $N_i\subset [0,\tau]$. This means that $\dot{\varphi}_i(t)=\lim_{h\to 0}\langle p_h(t),v_i \rangle$ for every $i\in \N$ and $t\in [0,\tau]\setminus \bigcup_{j\in \N}N_j$. Recalling that $\{ p_h(t)\mid h\ne 0 \}$ is relatively weakly$^*\!$ compact for every $t\in [0,\tau]\setminus N_0$, it has a subnet (which we do not relabel) that converges weakly$^*\!$ to an element in $E^*$. Therefore, $\dot{p}(t)=\mathit{w}^*\text{-}\lim_{h\to 0}p_h(t)$ exists for every $t\in [0,\tau]\setminus \bigcup_{i=0}^\infty N_i$ because $\{ v_i \}_{i\in \N}$ is a total family of $E$.
\end{proof}

\begin{proof}[Proof of Theorem \ref{thm4}]
Let $t\in \R_+$ be arbitrarily given and $\eta>0$ be as in $\mathrm{(H_8)}$. Take any $x^*\in \partial^-_xV(0,x_0(0))$ and let $f:\R_+\to E^*$ be a Gelfand integrable selector of the Dini--Hadamard superdifferential mapping $s\rightsquigarrow \partial^+_xL(s,x_0(s),\dot{x}_0(s))$, whose existence is guaranteed in Theorem \ref{thm1}. Define $p(t)=\int_0^tf(s)ds-x^*$ as a Gelfand integral. We claim that $-p(t)\in \partial^-_xV(t,x_0(t))$. To this end, fix any $v\in E$ and consider the local perturbation of $x_0(\cdot)$ over $[0,t]$ given by $x_\theta(s):=x_0(s)+\theta v$ for $s\in [0,t]$. By construction, $\dot{x}_\theta(s)=\dot{x}_0(s)$ a.e.\ $s\in [0,t]$ and $x_\theta(s)\in x_0(s)+\eta B$ whenever $0<\theta\le (1+\| v \|)^{-1}\eta$, and hence, $(x_\theta(s),\dot{x}_\theta(s))\in \mathrm{gph}\,\Gamma(s,\cdot)$ a.e.\ $s\in [0,t]$. By the Bellman principle of optimality, we have
$$
V(0,x_\theta(0))\le \int_0^tL(s,x_\theta(s),\dot{x}_\theta(s))ds+V(t,x_\theta(t)). 
$$
Subtracting $V(0,x_0(0))=\int_0^tL(s,x_0(s),\dot{x}_0(s))ds+V(t,x_0(t))$ from the both sides of the above inequality yields
\begin{align*}
V(0,x_\theta(0))-V(0,x_0(0))
& \le \int_0^t\left[ L(s,x_\theta(s),\dot{x}_\theta(s))-L(s,x_0(s),\dot{x}_0(s)) \right]ds \\
& \qquad +V(t,x_\theta(t))-V(t,x_0(t)).
\end{align*}
Let $\{ \theta_n \}_{n\in \N}$ be a sequence of positive real numbers with $\theta_n\to 0$ such that
$$
V_x^-(t,x_0(t);v)=\lim_{n\to \infty}\frac{V(t,x_0(t)+\theta_nv)-V(t,x_0(t))}{\theta_n}.
$$
Dividing the both sides of the above inequality by $\theta_n$ and taking the limit as $n\to \infty$ yields 
\begin{align*}
V^-_x(0,x_0(0);v)
& \le \int_0^tL^+_x(s,x_0(s),\dot{x}_0(s);v)ds+V^-_x(t,x_0(t);v) \\
& \le \int_0^t\langle f(s),v \rangle ds+V^-_x(t,x_0(t);v) \\
& =\langle p(t),v \rangle+\langle x^*,v \rangle+V^-_x(t,x_0(t);v)
\end{align*}
for every $v\in E$, where we employ the Lebesgue dominated convergence theorem and Fatou's lemma to derive that
\begin{align*}
& \limsup_{n\to \infty}\int_0^t\frac{L(s,x_0(s)+\theta_n v,\dot{x}_0(s))-L(s,x_0(s),\dot{x}_0(s))}{\theta_n}ds \\
{}\le{} 
& \int_0^tL^+_x(s,x_0(s),\dot{x}_0(s);v)ds.
\end{align*}
On the other hand, $\langle x^*,v \rangle\le V^-_x(0,x_0(0);v)$. Hence, $\langle -p(t),v \rangle\le V^-_x(t,x_0(t);v)$ for every $v\in E$ and thus our claim is true. 

Since $\langle p(t),y \rangle=\int_0^t\langle f(s),y \rangle ds-\langle x^*,y \rangle$ for every $t\in \R_+$ and $y\in E$ with $|\langle f(s),y \rangle|\le \| f(s) \| \| y \|\le l_1(s) \| y  ||$, we get $|\langle p(t+h)-p(t),y \rangle| \le \| y \|\int_t^{t+h} l_1(s)ds$, and therefore, $\| p(t+h)-p(t) \|\le \int_t^{t+h} l_1(s)ds$ for every $h>0$. This means that the function $p:\R_+\to E^*$ constructed above is locally absolutely continuous. In view of Lemma \ref{ko}, the weak$^*\!$ derivative $\dot{p}(t)=f(t)$ exists a.e.\ $t\in \R_+$. This demonstrates that the adjoint inclusions (i) and (iii) hold. Since Theorem \ref{thm3} and condition (i) yield $\langle p(t),\dot{x}_0(t) \rangle-L(t,x_0(t),\dot{x}_0(t))\ge \langle p(t),y \rangle-L(t,x_0(t),y)$ for every $y\in \Gamma(t,x_0(t))$, the maximum principle (iv) holds. Thus, for a.e.\ $t\in \R_+$ and every $v\in T_{\Gamma(t,x_0(t))}(\dot{x}_0(t))$, we have $\langle p(t),v \rangle \le L^-_y(t,x_0(t),\dot{x}_0(t);v)\le L^\circ_y(t,x_0(t),\dot{x}_0(t);v)$ and condition (ii) follows from the separation argument as in the proof of Theorem \ref{thm3}(ii). To verify the transversality condition (v) at infinity, recall that by Theorem \ref{thm2}, $V(t,\cdot)$ is Lipschitz of rank $k(t)$ with $k(t)\to 0$ as $t\to \infty$. Therefore, $\| p(t) \|\le k(t)\to 0$.  
\end{proof}

\begin{proof}[Proof of Theorem \ref{thm5}]
Let $t\in \R_+$ and $v\in E$. Consider the variational equation:
\begin{equation}
\label{eq4}
\dot{w}(s)=\nabla_xf(s,x_0(s),u_0(s))w(s) \ \text{a.e.\ $s\in [0,t]$}, \quad w(t)=v.
\end{equation}
In view of the separability of $E$ and $\mathrm{(H}_9)$, a unique mild solution $w(\cdot)\in W^{1,1}([0,t],E)$ to \eqref{eq4} satisfies 
$$
w(s)=v-\int_s^t\nabla_xf(\tau,x_0(\tau),u_0(\tau))w(\tau)d\tau \quad \text{for every $s\in [0,t]$}
$$
and its strong derivative $\dot{w}(s)$ exists a.e.\ $s\in [0,t]$ and satisfies \eqref{eq4} by the Lebesgue differentiation theorem. It follows from \citet[Theorem 4.2]{fr90} (applied to $F(s,x)\equiv \{ f(s,x,u_0(s)) \}$ and $A\equiv 0$) that for every $\theta>0$ there exists a mild solution $x_\theta(\cdot)$ to 
$$
\dot{x}(s)=f(s,x(s),u_0(s)) \ \text{a.e.\ $s\in [0,t]$}, \quad x(t)=x_0(t)+\theta v
$$
such that $(x_\theta(s)-x_0(s))/\theta \to w(s)$ uniformly in $s\in [0,t]$ as $\theta\to 0$. 

Take any $x^*\in \partial^-_xV(0,x_0(0))$ and let $g:\R_+\to E^*$ be a locally Bochner integrable selector from $s\rightsquigarrow \partial^+_x\tilde{L}(s,x_0(s),u_0(s))$. Consider the adjoint system:
\begin{equation}
\label{eq3}
-\dot{p}(s)=\nabla_xf(s,x_0(s),u_0(s))^*p(s)-g(s) \ \text{a.e.\ $s\in [0,t]$}, \quad p(0)=-x^*. 
\end{equation}
Since the mapping $(s,x,u)\mapsto \nabla_xf(s,x,u)^*$ has separable values in $E^*$ in view of $\mathrm{(H}_9)$, a unique mild solution to \eqref{eq3} does exist. As in the proof of Theorem \ref{thm4} (via the Bellman principle of optimality), we obtain the inequality 
\begin{align*}
& V(0,x_\theta(0))-V(0,x_0(0)) \\ {}\le{}
& \int_0^t\left[ L(s,x_\theta(s),f(s,x_\theta(s),u_0(s))-L(s,x_0(s),f(s,x_0(s),u_0(s)) \right]ds \\ 
& \quad +V(t,x_0(t)+\theta v)-V(t,x_0(t)).
\end{align*}
Divide the both sides of the above inequality by $\theta$ and let $\theta\to 0$ to get
$$
\langle x^*,w(0) \rangle \le \int_0^t\langle g(s),w(s) \rangle ds+V_x^-(t,x_0(t);v). 
$$
It follows from the a.e.\ strong differentiability of $p(\cdot)$ and $w(\cdot)$ that
\begin{align*}
\int_0^t\langle g(s),w(s) \rangle ds
& =\int_0^t \langle \dot{p}(s)+\nabla_xf(s,x_0(s),u_0(s))^*p(s),w(s) \rangle ds \\
& =\int_0^t \left[ \langle \dot{p}(s),w(s) \rangle+\langle p(s),\nabla_xf(s,x_0(s),u_0(s))w(s) \rangle \right]ds \\
& =\int_0^t \left[ \langle \dot{p}(s),w(s) \rangle+\langle p(s),\dot{w}(s) \rangle \right]ds \\
& =\int_0^t\frac{d}{ds}\langle p(s),w(s) \rangle ds=\langle p(t),v \rangle-\langle p(0),w(0) \rangle. 
\end{align*}
Hence, $\langle -p(t),v \rangle\le V_x^-(t,x_0(t);v)$ for every $v\in E$. This means that $-p(t)\in \partial^-_xV(t,x_0(t))$. This being true for every $t\in \R_+$, we deduce the adjoint inclusions (i) and (iii). The rest of the conditions follows as in  the proof of Theorem \ref{thm4}. 
\end{proof}

}
\end{document}